\documentclass{ws-m3as}

\usepackage{mathrsfs}
\usepackage{indentfirst}
\usepackage[pdftex,bookmarks,unicode,colorlinks]{hyperref}
\usepackage[numbers]{natbib} 
\usepackage{yfonts} 
\usepackage{stmaryrd} 

\allowdisplaybreaks[4] 
\newcommand{\E}{\mathbf E}
\renewcommand{\P}{\mathbf P}

\newcommand{\R}{\mathbb R}

\renewcommand{\L}{\mathcal L}
\newcommand{\e}{\epsilon}
\newcommand{\F}{\mathcal F}
\newcommand{\B}{\mathcal B}

\newcommand{\ind}{\mathbf 1}

\newcommand{\esssup}{\mathrm{ess\, sup}}

\begin{document}

\markboth{Q. Huang and X. Zhang}{Stochastic singular CS model}

%
\catchline{}{}{}{}{}
%

\title{ On the Stochastic Singular Cucker--Smale Model: Well-Posedness, Collision-Avoidance and Flocking}

\author{Qiao Huang}

\address{Department of Mathematics, Faculty of Sciences, University of Lisbon, \\ Campo Grande, Edif\'{\i}cio C6, PT-1749-016 Lisboa, Portugal.\\
qhuang@fc.ul.pt}

\author{Xiongtao Zhang}

\address{Center for Mathematical Sciences, Huazhong University of Science and Technology,\\ Wuhan, Hubei 430074, P.R. China.\\
xtzhang@hust.edu.cn}

\maketitle

\begin{history}
\received{(Day Month Year)}
\revised{(Day Month Year)}
\comby{(xxxxxxxxxx)}
\end{history}

\begin{abstract}
  We study the Cucker--Smale (C-S) flocking systems involving both singularity and noise. We first show the local strong well-posedness for the stochastic singular C-S systems before the first collision time, which is a well defined stopping time.
  Then, for communication with higher order singularity at origin (corresponding to $\alpha\ge1$ in the case of $\psi(r)=r^{-\alpha}$), we establish the global well-posedness by showing the collision-avoidance in finite time, provided that there is no initial collisions and the initial velocities have finite moment of any positive order. Finally, we study the large time behavior of the solution when $\psi$ is of zero lower bound, and provide the emergence of conditional flocking or unconditional flocking in the mean sense, for constant and square integrable intensity respectively.
  \end{abstract}



\section{Introduction}\label{sec:1}
\setcounter{equation}{0}
Collective behaviors are ubiquitous in our daily life, such as flocking of birds, swarming of fish, synchronization of fireflies, \cite{TT98,TB04,VCBJ06,Win67} etc. In order to study emergence of collective behaviors, various multi-agent systems have been established, for instance, Winfree model \cite{Win67}, Kuramoto model \cite{Kur75}, Cucker--Smale (C-S) model \cite{CS07B}, etc. In recent decades, these seminal models have been extensively studied and applied to many different areas, such as bacteria aggregation, UAV formation, consensus in social network, etc. In this paper, we will study the C-S model with singular communication weight and white noise (see Section \ref{sec:2} for the detailed formulation). More precisely, let $(\Omega, \F, \P, \{\F_t\}_{t\ge0})$ be a filtered probability space satisfying the usual conditions. We consider the following stochastic Cucker--Smale model,
\begin{equation}\label{CS}\left\{
  \begin{aligned}
    dx_i(t) &= v_i(t) dt, \\
    dv_i(t) &= \frac{\lambda}{N}\sum_{j=1}^N \psi(|x_i(t)-x_j(t)|)(v_j(t)-v_i(t))dt + D(t) v_i(t) dW(t), \\
    (x_i(0)&, v_i(0)) = (x^0_{i},v^0_{i}) \in \R^{2d}, \quad i=1,2,\cdots,N,
  \end{aligned} \right.
\end{equation}
subject to that the initial law is given by a probability measure $\mu_0$ on $(\R^{2Nd},\B(\R^{2Nd}))$, namely,
\begin{equation*}
  \L(x^0,v^0) = \mu_0,
\end{equation*}
where $(x_i,v_i)$ denotes the position-velocity pair of $i$-th particle, $\psi$ is a communication weight that is allowed to have a singularity at the origin, $W=\{W(t)\}_{t\ge0}$ is an one-dimensional standard Brownian motion with respect to the filtration $\{\F_t\}_{t\ge0}$, $\lambda$ is a positive constant that indicates the coupling strength, $D:[0,\infty)\to \R$ is the noise intensity. Moreover, the initial data $(x^0,v^0)$ is an $\F_0$-measurable $\R^{2Nd}$-valued random vector independent of $W$ (where, of course, the probability space $(\Omega, \F, \P, \{\F_t\}_{t\ge0})$ is assumed to be rich enough to accommodate a random vector $(x^0,v^0)$ independent of $W$, cf. \cite[Section 5.2]{KS91}). We make a final assumption for the initialization, which will be clarified in the next section, that the joint distribution $\mu_0$ of $(x^0,v^0)$ satisfies
\begin{equation}\label{initial-cond}
  \mu_0\left( \left\{(x,v)\in \R^{2Nd}: \sum_{i=1}^N x_i = 0, \quad \sum_{i=1}^N v_i = 0 \right\}\right) = 1.
\end{equation}

The C-S model was firstly introduced by F. Cucker and S. Smale in 2007 when they studied the flocks of birds \cite{CS07B}. Since then, numerous of works related to this model have been produced in various areas and applications. To name a few, mono-cluster and multi-cluster formation \cite{BH17,CHH16,HKPa19}, avoidance of collision \cite{CD11}, mean-field limit and kinetic C-S model \cite{ABF19,CCR11,CFR10}, hydrodynamic limit and hydrodynamic C-S model \cite{DM07,HT08,KPS21,PS17,TT14}, random environment case \cite{AH10,CS19,CM08}, etc. Readers are also referred to \cite{FST16,ST20} for other related aggregation and flocking models.

Comparing to the original regular communication setting, the singular communication function and random noise setting are more proper and important in many aspects. In fact,  the noise makes the model more applicable since there are a lot of uncertainty in the real environment, while the singular interaction offers repulsion mechanism to avoid collisions, which is very important in real applications such as UAV formation. However, the combination of singularity and the noise draws a lot of difficulty in theoretical analysis, since the classical theories generally cannot be directly applied to yield the well-posedness, collision behaviors and large time behaviors for singular systems. Therefore, comparing to the regular case, there are only a few literature and results on the singular communications in both deterministic and stochastic cases. It is worth remarking here that due to the pattern of noise in \eqref{CS} that the intensity $D$ is varying in time, such stochastic C-S model belongs to a wide class of stochastic models, called \emph{time-dependent diffusion models}, which appears mostly in the mathematical finance \cite{FJZ03,HW90}. Compared to the stochastic models with constant intensity which is often called time-homogeneous models, the time-dependent diffusion models allow the instantaneous noise return to evolve with time. This enables one to simultaneously capture the time effect and reduce modeling bias. A specific example for time-dependent diffusion models will be provided in Remark \ref{remark-4}.

The well-posedness of the deterministic and stochastic singular C-S model is mainly affected by the singularity of the interaction, and the studies mainly focus on the deterministic case so far. The well-posedness for the deterministic case has been established in \cite{CCH14, MP18, Pes14, Pes15}. Then, the authors in \cite{CCM17} further constructed a sharp condition for the collision-avoidance. The paper \cite{HKP19} provided a rigorous proof of the mean-field limit and propagation of chaos. Recently, the authors in \cite{ZZ20} studied the deterministic one-dimensional case and obtained the complete classification of asymptotic behaviors. While for the stochastic case, the authors in \cite{CDP18} studied the well-posedness of stochastic C-S models with regular communication weights and various patterns of noise. We also refer to \cite{BCC11} where the authors studied the well-poseness and propagation of chaos for a class of stochastic systems of interacting particles with non-Lipschitz forces.

On the other hand, the large time behavior mainly depends on the noise intensity and the structure of $\psi$ at infinity, thus most of the studies concern on the stochastic C-S model with regular interaction. The authors in \cite{AH10} studied the regular C-S model with a multiplicative noise with constant intensity, and provided the emergence of pathwise flocking (or flocking in orbit sense). In fact, the pathwise flocking can be guaranteed directly by the multiplicative noise due to the pathwise decay of the exponential martingale of Brownian motion, and thus the paper \cite{AH10} announced that the multiplicative noise enhances the flocking emergence. Later on, \cite{CS19} studied a kinetic model that is similar to the kinetic version of \cite{AH10}, and provided the rigorous proof of the mean-field limit and large time behavior. Unlike \cite{AH10}, the authors in \cite{CS19} considered the flocking in a somewhat stronger sense, which requires asymptotic vanishing of the kinetic energy. For this reason, they announced a phase change phenomenon from a non-flocking to a flocking state depending on the intensity of noise, among which the flocking requires positive lower bound of $\psi$ and small noise. The paper \cite{TLY14} is just devoted to the discrete model associated to the kinetic model in \cite{CS19}. Thus the flocking they considered is in the sense of mean-square, as the discrete counterpart of the flocking in \cite{CS19}. Recently, the authors in \cite{CDP18} systematically discussed the pathwise unconditional and conditional flocking as well as the mean-square flocking, for the stochastic C-S models of different variants of noise. Actually, they compared different descriptions on flocking and mainly focus on the flocking in mean-square sense. Based on these existing results and observations, there are three key and natural questions or difficulties for the stochastic singular C-S model:

\begin{itemize}
\item Can we establish the well-posedness theory for the stochastic singular C-S system?
\item Can we describe the collision behaviors as in deterministic case. In particular, can we show the collision-avoidance for higher order singular communications?
\item How does the solution behave in large time? What kind of description for the flocking is better? And will the random noise make it more difficult to flock as mentioned in \cite{CS19} or, on the contrary, enhance the flocking as in \cite{AH10}?
\end{itemize}

To answer the above questions, the strategy and results in the present paper are three folds. First, we will show the local well-posedness of strong solutions to the C-S model with general communication weight (which could be singular). In this case, collisions may occur in finite time if the communication is not singular enough near the origin. And so far, we do not know how to properly restart the evolution of the system after a collision. Therefore we can only prove the unique solution exists before the first collision time. See Theorem \ref{local} and Remark \ref{remark-2} for details.

Second, we will show the collision-avoidance with probability one when the order of singularity is high enough (Theorem \ref{coll-avoid} and Remark \ref{remark-3}), thus the global well-posedness of strong solutions follows immediately .

Finally, to study the large time behavior for the system, we first follow the settings in \cite{CS19,TLY14} to assume a positive lower bound for the communication function, which provides the spectrum gap and thus generates exponentially fast flocking. We emphasize that, by our definition of flocking in Definition \ref{def-flocking}, there is no phase transition phenomenon mentioned in \cite{CS19,TLY14} (see Remark \ref{R2.3} and Proposition \ref{flocking-1}). This also coincides with a result in \cite{CDP18} of unconditional $L^1$ flocking for the case of positive constant communications.
Then, we will further consider the case when the communication has \emph{zero lower bound}. In this case, we do not have uniform dissipation structure to control the noise, and thus the exponential martingale will play a very important role. Actually, we find that the time asymptotical property of the intensity is crucial for studying large time behavior. So we will assume the noise intensity to be a constant or square integrable, and study the two cases respectively. It turns out that, technically, the stronger nonzero constant noise is better for aggregation analysis, since the exponential martingale of Brownian motion (see Definition \ref{exponential martingale}) is a.s. integrable in time. While the square integrable intensity is technically better for velocity alignment estimates, since the noise will disappear asymptotically. This is also a reason why we adpot the general setting of time-varing intensity in the present paper. Based on these observations, we first define the conditional flocking and unconditional flocking for the stochastic case (Definition \ref{def-flocking} and Definition \ref{def-cond-flocking}), and provide the emergence results of conditional and unconditional flocking respectively (Theorem \ref{cond-flock}, Theorem \ref{uncond-flock-constant-D} and Theorem \ref{flocking-integrable}). Note that these methods and results all work for regular communication case, since the large time behavior mainly relates to the far field structure of communications.

The rest of the paper will be organized as follows. In Section \ref{sec:2}, we will introduce some preliminary concepts and a priori results. In Section \ref{sec:3}, we will prove the local well-posedness for the system \eqref{CS} with general singular communication weight $\psi$. In Section \ref{sec:4}, we will consider the higher order singular case which includes $\psi(r) = \frac{1}{r^{\alpha}}$ with $\alpha\ge1$. We will prove the almost sure collision-avoidance for this case and hence the global well-posedness follows. Section \ref{sec:5} is devoted to the  large time behavior of the system, with nonzero constant intensity and square integrable intensity respectively. Finally, Section \ref{sec:6} will be contributed as a short summary and limited discussion.

Throughout this paper, the small letter $c$ is reserved to denote a finite \emph{positive} constant whose value may vary from line to line. We will use the notation $c(\cdots)$ to emphasize the dependence on the quantities appearing in the parentheses. We will occasionally omit the time variable $t$ (or $s$) for convenience, and $D$ is viewed to be a function of time $t$ unless otherwise specified. In addition, we let $|\cdot|$ be the $l^2$-norm on $\R^d$ and set for $0<p<\infty$ that,
\begin{equation*}
  \|x\|_p = \left( \sum_{i=1}^N |x_i|^p \right)^{1/p}, \quad \|v\|_p = \left( \sum_{i=1}^N |v_i|^p \right)^{1/p}.
\end{equation*}

\section{Preliminary}\label{sec:2}
\setcounter{equation}{0}

In this section, we will first briefly show the formulation of the system \eqref{CS} from the original C-S model, then we will introduce some basic notions and a priori results that will be used in later sections.

\subsection{Formulation of the model}\label{subsec:2-1}
In this part, we will introduce the connection between the system \eqref{CS} and the original system initiated by F. Cucker and S. Smale in their seminal work \cite{CS07B} in 2007. Actually, the authors in \cite{CS07B} studied the following model,
\begin{equation*}
\left\{
  \begin{aligned}
   \frac{d \tilde x_i}{dt} &= \tilde v_i , \\
    \frac{d \tilde v_i}{dt} &= \frac{\lambda}{N}\sum_{j=1}^N \bar{\psi}(|\tilde x_i-\tilde x_j|)(\tilde v_j-\tilde v_i). \\
  \end{aligned} \right.
\end{equation*}
where $(\tilde x_i,\tilde v_i)$ denotes the position and velocity of $i$-th particle at time $t$, and $\bar{\psi}$ is the communication weight. Now, we consider $\bar{\psi}$ as a random perturbation of a deterministic communication function $\psi$. More precisely, we set
\[\bar{\psi}=\psi-\frac{D}{\lambda}\dot W,\]
where $W$ is a scalar Brownian motion and $\dot W$ is the associated white noise, $D$ is a time-dependent noise intensity and the negative sign is just for convenience. Then, we obtain a stochastic C-S model as below,
\begin{equation}\label{B1}
\left\{
  \begin{aligned}
    d\tilde x_i &= \tilde v_i dt, \\
    d\tilde v_i &= \frac{\lambda}{N}\sum_{j=1}^N \psi(|\tilde x_i-\tilde x_j|)(\tilde v_j-\tilde v_i)dt - \frac{D}{N}\sum_{j=1}^N (\tilde v_j-\tilde v_i)dW. \\
  \end{aligned} \right.
\end{equation}
Now we follow \cite{AH10,HL09} to introduce  macroscopic quantities (center of mass) and microscopic quantities (fluctuations) for the system \eqref{B1}:
\begin{alignat}{2}
  \bar x &= \frac{1}{N}\sum_{j=1}^N \tilde x_j, &\hspace{20pt} \bar v &= \frac{1}{N}\sum_{j=1}^N \tilde v_j, \label{macro} \\
  x_i &= \tilde x_i - \bar x, & v_i &= \tilde v_i - \bar v. \label{micro}
\end{alignat}
It is easy to derive from \eqref{B1} and \eqref{macro} that the center of mass satisfies
\begin{equation}\label{center}
  \frac{d\bar x}{dt} = \bar v, \quad \frac{d\bar v}{dt} = 0,
\end{equation}
which can be uniquely solved by $\bar x(t) = \bar x(0) + \bar v(0) t$, $\bar v(t) = \bar v(0)$. Subtracting \eqref{center} from \eqref{B1}, we conclude that the microscopic part $(x,v)$ satisfies the equation \eqref{CS}. The constraint \eqref{initial-cond} for initial data follows immediately from the relation \eqref{micro}. Indeed, \eqref{initial-cond} means that with probability one, the initial data $(x(0),v(0))$ is centered, i.e., $\sum_{i=1}^N x_i(0) = 0$ and $\sum_{i=1}^N v_i(0) = 0$. 
Now, if we sum up \eqref{CS} over $i=1,2,\cdots,N$, we obtain a system of equations for $\sum_{i=1}^N x_i$ and $\sum_{i=1}^N v_i$ as follows:
\begin{equation*}
  d\sum_{i=1}^N x_i = \sum_{i=1}^N v_i dt, \quad d\sum_{i=1}^N v_i = D \sum_{i=1}^N v_i dW,
\end{equation*}
which together with the constraint \eqref{initial-cond} implies that with probability one, the following conservation laws hold for every $t\ge0$,
\begin{equation}\label{center-all}
  \sum_{i=1}^N x_i(t) = 0,\quad \sum_{i=1}^N v_i(t) = 0.
\end{equation}
Note that this is also a direct consequence of \eqref{micro}. Therefore, system \eqref{CS} is actually identical to the original system \eqref{B1}.


\subsection{Definitions of solutions and flocking}
In this part, we will introduce some basic concepts that will be frequently used in later sections. First, we will define the local strong solution to the system \eqref{CS}.

\begin{definition}[Strong solution up to a stopping time]\label{def-local}
  Let $\tau$ be an $\{\F_t\}$-stopping time. A continuous, $\{\F_t\}$-adapted, $\R^{2Nd}$-valued process $(x,v)$ defined on the stochastic time interval $[0,\tau)$ is called a \emph{strong solution up to $\tau$} to the system \eqref{CS} if there exists a non-decreasing sequence of $\{\F_t\}$-stopping times $\{\tau_n\}_{n\ge1}$ such that
  \begin{itemize}
    \item[(i)] $0\le\tau_n \le\tau$ for all $n\ge1$ and $\lim_{n\to\infty} \tau_n = \tau$ a.s.,
    \item[(ii)] $\tau_n<\tau$ for all $n\ge1$ a.s. on the event $\{0<\tau<\infty\}$,
    \item[(iii)] for each $n\ge1$,
  \begin{equation*}
    \begin{split}
      \P\Bigg( x_i(t\wedge\tau_n) &= x_i(0) + \int_0^{t\wedge \tau_n} v_i(s) ds, \\
      v_i(t\wedge\tau_n) & = v_i(0) + \int_0^{t\wedge \tau_n} D(s) v_i(s)dW(s)\\
       &\quad + \frac{\lambda}{N}\sum_{j=1}^N \int_0^{t\wedge \tau_n} \psi(|x_i(s)-x_j(s)|)(v_j(s)-v_i(s))ds, \\
      x_i(t\wedge\tau_n) &\ne x_j(t\wedge\tau_n), \quad\forall 1\le i\ne j \le N, \forall 0\le t<\infty  \Bigg) = 1.
    \end{split}
  \end{equation*}
  \end{itemize}
\end{definition}

\begin{remark}\label{remark-5}
  It follows from (i) and (iii) in Definition \ref{def-local} that a.s., $x_i(t) \ne x_j(t)$ for all $t\in[0,\tau)$ and $i\ne j$.
\end{remark}

Then, we introduce the concept of exponential martingale, which will be frequently used in the later estimates about the stochastic integral, and plays a very important role in the time-asymptotic analysis.
\begin{definition}[Exponential martingale]\label{exponential martingale}
Let $D$ be a real-valued function on $[0,\infty)$ satisfying $\int_0^T D^2(t)dt <\infty$ for all $T>0$. By the \emph{exponential martingale} of the stochastic integral process $\int_0^\cdot D(s) dW(s)$, we mean a process $\mathcal E = \{\mathcal E(t)\}_{t\ge0}$ defined by,
\[\mathcal E(t):=\exp\left(-\frac{1}{2}\int_0^t D^2(s) ds + \int_0^t D(s)dW(s) \right), \quad t\ge0.\]
\end{definition}

\begin{remark}
  The exponential martingale $\mathcal E$ is indeed a martingale due to the Novikov's criterion \cite[Corollary 3.5.13]{KS91}. If we enhance the integrable assumption for $D$ to $\int_0^\infty D^2(t)dt <\infty$, then $\mathcal E$ is also uniformly integrable \cite[Theorem III.45]{Pro05}. The latter assumption will be used when studying the flocking in Subsection \ref{subsec:5-2}. In any case, we have $\E(\mathcal E(t)) \equiv \E(\mathcal E(0)) = 1$. Actually, the exponential martingale we defined here is a very special case of the so called Dol\'eans-Dade exponential, of which we refer to \cite{JS13,KS91} for more details.
\end{remark}

We next extend the definition of time-asymptotic flocking in \cite{HL09} to the stochastic case. Generally speaking, the emergence of time-asymptotic flocking means that the velocity of all agents tends to a common value and the relative position of them remains bounded (both in certain probability sense) when time tends to infinity. Then, recalling the notations in \eqref{macro} and \eqref{micro}, we define the (unconditional) flocking and conditional flocking as follows (cf. \cite[Definition 1.10.(2)]{CDP18}).
\begin{definition}[Flocking in mean]\label{def-flocking}
  Let $p\ge2$. We say the stochastic C-S system \eqref{B1} has a \emph{time-asymptotic flocking in mean}, if the solution $(\tilde x,\tilde v)$ satisfies
  \begin{itemize}
    \item[(i)] $\lim_{t\to\infty} \E(\|\tilde v(t)- \bar v(t)\|_p) = 0$ (velocity alignment),
    \item[(ii)] $\sup_{t\ge0} \E(\|\tilde x(t)- \bar x(t)\|_p) <\infty$ (aggregation or group forming).
  \end{itemize}
\end{definition}

Since as we showed in subsection \ref{subsec:2-1}, the original stochastic C-S system \eqref{B1} are related to the reduced system \eqref{CS} by $x = \tilde x-\bar x$ and $v = \tilde v-\bar v$, the velocity alignment condition in Definition \ref{def-flocking} is then equivalent to $\lim_{t\to\infty} \E(\|v(t)\|_p) = 0$, and the group forming condition is equivalent to $\sup_{t\ge0} \E(\|x(t)\|_p) <\infty$.\newline

\begin{remark}\label{R2.3}
There are several kinds of definitions for the time-asymptotic flocking. We make a brief comparison here.

(i). The authors in \cite{AH10} (and also in \cite{CDP18}) used the \emph{almost sure convergence} $\text{a.s.-}\lim_{t\to\infty} \|\tilde v(t)- \bar v(t)\|_2 = 0$ to describe the pathwise flocking. In fact, they found that the multiplicative noise can generate an exponential martingale (see Proposition \ref{regular} below for a sketch), which directly leads to the pathwise flocking even without the coupling of communications in drift terms. However, when we take expectation to the exponential martingale, we will obtain a constant since it is a martingale. Therefore, 
the multiplicative noise can only enhance the flocking in pathwise sense, but not in mean sense.

(ii).  The papers \cite{CS19,TLY14} (and also \cite{CDP18}) used the \emph{mean-square convergence} $\lim_{t\to\infty} \E(\|\tilde v(t)- \bar v(t)\|^2_2) = 0$ to describe the flocking emergence, and thus obtained the flocking when the lower bound of the communication function is greater than the noise intensity. This in fact implies the phase transition from energy dissipation to energy increasing in the system, which cannot be avoided due to the positive noise intensity. However, there will be no such phase transition phenomenon when we adopt Definition \ref{def-flocking} to describe the flocking, and in this situation, the energy dissipation is only sufficient but not necessary to the flocking emergence. 
The reason is due to the martingale nature of the exponential martingale. Intuitively, as the exponential martingale has constant expectation, the exponential decay induced by the communication part dominates. But if we take the mean-square to the exponential martingale, it will lead to an exponential growth, and thus the flocking occurs only if the lower bound of the communication can dominate the additional growth.

(iii). Note that all these definitions of flocking are equivalent in the deterministic case, but not in the stochastic case. According to the discussions above, we may say it is better to use Definition \ref{def-flocking} (i.e., $L^1$ flocking in \cite{CDP18}) to describe the emergence of flocking for the stochastic C-S model \eqref{CS}.
\end{remark}

\begin{definition}[Conditional flocking in mean]\label{def-cond-flocking}
  Let $p\ge2$. Let $A\in\F$ be an event with nonzero probability, i.e., $\P(A)>0$. We say the stochastic C-S system \eqref{CS} (or \eqref{B1}) has a \emph{time-asymptotic flocking in mean conditioning on $A$} or \emph{conditional time-asymptotic flocking in mean given $A$}, if the solution $(x,v)$ satisfies
  \begin{itemize}
    \item[(i)] $\lim_{t\to\infty} \E(\|v(t) \|_p | A) = 0$ (conditional velocity alignment),
    \item[(ii)] $\sup_{t\ge0} \E(\|x(t) \|_p | A) <\infty$ (conditional aggregation or group forming).
  \end{itemize}
\end{definition}

\begin{remark}\label{R2.6}
We make a few remarks for the definition of conditional flocking.

(i). Logically, the dynamical system \eqref{CS} has ``input'' the initial data $(x^0,v^0)$ and the noise $W$, and ``output'' $(x(t),v(t))$. The principle of causality for dynamical systems (cf. \cite[Section 5.2]{KS91}) requires naturally that the conditioning event $A$ needs be only related to the ``input'', that is, $A$ needs be taken in the $\sigma$-algebra $\F_0 \vee \F^W_\infty$, where $\F^W_\infty$ is the $\sigma$-algebra generated by the Brownian motion $W$ on whole time interval $[0,\infty)$, that is, $\F^W_\infty = \sigma\{W(t); 0\le t <\infty\}$.

(ii). Obviously, if the conditioning event $A$ has full probability, i.e., $\P(A)= 1$, then the flocking conditioning on $A$ coincides with the (unconditional) flocking defined in Definition \ref{def-flocking}.

(iii). If there is no noise in the system \eqref{CS}, and the initial distribution $\mu_0$ is set up as an empirical distribution $\mu_0 = \frac{1}{n}\sum_{j=1}^n \delta_{(x^{0,j},v^{0,j})}$, for given pairs $(x^{0,j},v^{0,j})\in\R^{2Nd}$, $j=1,2,\cdots,n$ which are the values of $n$ independent observations in a sample, then the system \eqref{CS} degenerates to $n$ separated deterministic C-S systems. In this case, the conditional flocking defined above reduces to the one used in \cite[Section 4]{HL09}, since the conditional expectation is just the normalized summation over the indices $j$'s that fulfill the given condition.

(iv). As we have said, the definitions for the flocking could be quite different. In \cite{CDP18}, the authors introduced a notion of pathwise conditional flocking, which applies well for a class of stochastic C-S models. While for the specific model \eqref{CS}, the unconditional pathwise flocking has been obtained in \cite{AH10}. In the present paper, we will focus on the conditional flocking in the mean sense. A simple example is provided in \ref{App-A} to show that the flocking in mean may not emerge conditioning on some particular events.
\end{remark}

\subsection{A priori results}

In this part, we introduce the well-posedness for the system \eqref{CS} when $\psi$ is regular, which can be viewed as an a priori estimates for later analysis. The proof is similar to \cite[Theorem 2.1]{TLY14}, and the $l^p$ estimates are standard generalization of the $l^2$ estimates in \cite[Lemma 3.1, 3.2]{AH10}. We will briefly sketch the proof for the reader's convenience.
\begin{proposition}\label{regular}
  Let $\psi, D:[0,\infty)\to(-\infty,\infty)$ be two locally Lipschitz functions. Assume $\psi$ is bounded from below, that is, $\psi_*:=\inf_{r\ge0}\psi(r)>-\infty$. Then for any probability measure $\mu_0$ on $(\R^{2Nd},\B(\R^{2Nd}))$, the system \eqref{CS} has a unique global strong solution $(x,v)$. 
  Moreover, with probability one, we have for each $p\ge 2$ and all $t\ge0$,
  \begin{align}
    &\left|\frac{d\|x(t)\|_p}{dt}\right| \le \|v(t)\|_p,  \label{apriori-1}\\
    & \|v(t)\|_p \le \|v(0)\|_p \exp\left(-\lambda\psi_* t -\frac{1}{2}\int_0^t D^2(s) ds + \int_0^t D(s)dW(s)\right).  \label{apriori-2}
  \end{align}
\end{proposition}

\begin{proof}
  Since all coefficients in \eqref{CS} are locally Lipschitz on $\R^{2Nd}$, the classical theory of the well-posedness for SDEs yields the existence of the unique strong solution up to the explosion time (see, e.g., \cite[Theorem IV.3.1]{IW89}). In the sequel, we will prove the two estimates \eqref{apriori-1} and \eqref{apriori-2} first, then it implies that with probability one the explosion time is infinity, which means the unique strong solution is global. Now, we use Cauchy--Schwarz inequality to get
  \begin{equation}\label{est-20}
    \begin{split}
      p \|x\|_p^{p-1} \left|\frac{d\|x\|_p}{dt}\right| &= \left|\frac{d\|x\|_p^p}{dt}\right| = p \left|\sum_{i=1}^N |x_i|^{p-2} \langle x_i,v_i\rangle \right| \\
      &\le p \left(\sum_{i=1}^N \left| |x_i|^{p-2} x_i \right|^{p'} \right)^{1/p'} \left( \sum_{i=1}^N |v_i|^p \right)^{1/p} = p \|x\|_p^{p-1} \|v\|_p,
    \end{split}
  \end{equation}
  which gives \eqref{apriori-1}. We turn to prove \eqref{apriori-2}. 
  Applying It\^o's formula to the equation of velocity $v_i$, we obtain the following equation,
  \begin{equation*}
    \begin{split}
      d|v_i|^2 &= 2 \langle v_i, dv_i \rangle + \langle dv_i, dv_i \rangle \\
      &= \left( \frac{2\lambda}{N}\sum_{j=1}^N \psi(|x_i-x_j|)\langle v_i, v_j-v_i \rangle + D^2 |v_i|^2 \right) dt + 2D |v_i|^2 dW.
    \end{split}
  \end{equation*}
 Next, we apply It\^o's formula again and use the above formula for $|v_i|^2$ to obtain
  \begin{equation*}
    \begin{split}
      d \|v\|_p^p =&\ \sum_{i=1}^N d|v_i|^p = \sum_{i=1}^N \left( \frac{p}{2} |v_i|^{p-2} d|v_i|^2 + \frac{p}{4}\left( \frac{p}{2} - 1 \right) |v_i|^{p-4} d[|v_i|^2, |v_i|^2] \right) \\
      =&\ \sum_{i=1}^N \left[ \left( \frac{p\lambda}{N}\sum_{j=1}^N \psi(|x_i-x_j|) |v_i|^{p-2} \langle v_i, v_j-v_i \rangle + \frac{p D^2}{2} |v_i|^p \right) dt + pD |v_i|^p dW \right] \\
      &\ + \sum_{i=1}^N \frac{p(p-2)D^2}{2} |v_i|^p dt \\
      =&\ -\frac{p\lambda}{2N} \sum_{i=1}^N \sum_{j=1}^N \psi(|x_i-x_j|) \langle |v_i|^{p-2} v_i - |v_j|^{p-2} v_j, v_i-v_j \rangle dt \\
      &\ + \frac{p(p-1)D^2}{2} \|v\|_p^p dt + pD \|v\|_p^p dW,
    \end{split}
  \end{equation*}
  where the bracket $[\cdot,\cdot]$ on the right hand side (RHS) of the first equality denotes the quadratic variation.
  Note that since $p\ge2$, we have
  \begin{equation*}
    \begin{split}
      &\ \langle |v_i|^{p-2} v_i - |v_j|^{p-2} v_j, v_i-v_j \rangle \\
      =&\ |v_i|^p + |v_j|^p - \left( |v_i|^{p-2} + |v_j|^{p-2} \right) \langle v_i, v_j \rangle \\
      \ge&\ |v_i|^p + |v_j|^p - \frac{1}{2} \left( |v_i|^{p-2} + |v_j|^{p-2} \right) \left( |v_i|^2 + |v_j|^2 \right) \\
      =&\ \frac{1}{2} \left( |v_i|^{p-2} - |v_j|^{p-2} \right) \left( |v_i|^2 - |v_j|^2 \right) \\
      \ge&\ 0.
    \end{split}
  \end{equation*}
  This together with \eqref{center-all} yields
  \begin{align*}
      &\ \sum_{i=1}^N \sum_{j=1}^N \psi(|x_i-x_j|) \langle |v_i|^{p-2} v_i - |v_j|^{p-2} v_j, v_i-v_j \rangle \\
      \ge&\ \psi_* \sum_{i=1}^N \sum_{j=1}^N \langle |v_i|^{p-2} v_i - |v_j|^{p-2} v_j, v_i-v_j \rangle \\
      = &\ \psi_* \sum_{i=1}^N \sum_{j=1}^N \left[ |v_i|^p + |v_j|^p - \left( |v_i|^{p-2} + |v_j|^{p-2} \right) \langle v_i, v_j \rangle \right] \\
      = &\ \psi_* \left[ 2 N \|v\|_p^p - 2 \left\langle \sum_{i=1}^N |v_i|^{p-2}v_i, \sum_{j=1}^N v_j \right\rangle \right] = 2N \psi_* \|v\|_p^p.
  \end{align*}
  Hence, a straightforward application of the comparison theorem for one-dimensional SDEs (see, e.g., \cite[Theorem VI.1.1]{IW89}) yields that, $\|v(t)\|^p_p \le V(t)$ for $t\ge0$ with probability one, where $V$ satisfies the following SDE
  \begin{equation}\label{GBM}
    dV(t) = \left[ \frac{p(p-1)D^2(t)}{2} -p\lambda\psi_* \right] V(t) dt + pD(t) V(t) dW(t), \quad V(0) = \|v(0)\|^p_p,
  \end{equation}
  whose solution is given explicitly by
  $$V(t) = \exp\left(-p\lambda\psi_* t - \frac{p}{2}\int_0^t D^2(s) ds + p\int_0^t D(s)dW(s)\right).$$
  The local boundedness of $D$ ensures that the stochastic integral $\int_0^t D(s)dW(s)$ is well-defined for all $0\le t<\infty$. This proves \eqref{apriori-2}. And therefore, there is no finite time blow-up for $(x,v)$ and we finish the proof.
\end{proof}
$\ $

\begin{remark}\label{remark-4}
  (i). The SDE \eqref{GBM} for the dominating process $V$ is called the \emph{local volatility model} with instantaneous risk-free rate $\frac{1}{2}p(p-1)D^2(t)$, dividend yield $p\lambda\psi_*$ and instantaneous volatility $pD(t)$ \cite{Dup94}. There are two special cases for this SDE. In the case that $D$ is a constant, the process $V$ is the so called \emph{geometric Brownian motion} \cite{Oks03}. In the case $\lambda=0$, $V^{1/p}$ is just the exponential martingale $\mathcal E$ defined in Definition \ref{exponential martingale}. 

  (ii). In general, the stochastic integral process $\{\int_0^t \|v(s)\|_p dW(s)\}_{t\ge 0}$ is only a martingale but not uniformly integrable. When $D$ is a constant, a sufficient condition for this stochastic integral to be uniformly integrable is that $\E\left( \|v(0)\|_p^2 \right)<\infty$ and $D^2<2\lambda\psi_*$. Indeed, by \eqref{apriori-2} and the independence of $v(0)$ and $W$, we have
  \begin{equation*}
    \begin{split}
      \int_0^\infty \E \left( \|v(t)\|_p^2 \right) dt &\leq \E\left( \|v(0)\|_p^2 \right) \int_0^\infty \E\left[ \exp\left(-2\lambda\psi_* t -D^2t+2DW(t)\right) \right] dt \\
      &= \E\left( \|v(0)\|_p^2 \right) \int_0^\infty \exp\left(-2\lambda\psi_* t +D^2t \right) dt \\
      &<\infty.
    \end{split}
  \end{equation*}
  Note that this estimate coincides with the results in \cite{CS19,TLY14} that the mean-square flocking cannot emergent when $\psi_*\leq \frac{D^2}{2\lambda}$, as discussed in Remark \ref{R2.3}.
\end{remark}

\section{Local well-posedness for singular systems}\label{sec:3}
\setcounter{equation}{0}
In this section, we consider the stochastic C-S model with a general singular communication weight. More precisely, we define the communication weight $\psi$  only on the open interval $(0,\infty)$. And we will always assume in the sequel that $\psi_*:=\inf_{r>0}\psi(r)\ge0$ for convenience and avoiding misunderstandings. But actually for the purpose of local well-posedness and collision-avoidance, the requirement can be relaxed. Please see Remark \ref{remark-2}.(vi) for details.

As we discussed before in the introduction, collisions may occur if the order of singularity is low. Therefore, we first prove the local well-posedness of system \eqref{CS}, which basically means that the solution is always well-defined before the first collision occurs.

\begin{theorem}[Local well-posedness]\label{local}
  Let $\psi:(0,\infty)\to[0,\infty)$ and $D:[0,\infty)\to(-\infty,\infty)$ be two locally Lipschitz functions.
  Suppose the initial probability measure $\mu_0$ is collisionless, that is,
  \begin{equation}\label{initial-dist-no-coll}
    \mu_0\left( x_i \ne x_j, \forall 1\le i\ne j\le N \right) = 1.
  \end{equation}
  Then there exists a unique stopping time $\tau^*>0$ a.s. and a strong solution $(x,v)$ up to $\tau^*$ to the system \eqref{CS} (in the sense of Definition \ref{def-local}), such that 
  the following holds a.s.:
  \begin{equation}\label{collision-limit}
    \liminf_{t\uparrow\tau^*} \min_{i\ne j}|x_i(t)-x_j(t)| = 0, \quad\text{on } \{\tau^*<\infty\}.
  \end{equation}
  The uniqueness holds in the following sense: if $(y,u)$ is another strong solution up to a stopping time $\sigma$, then $\sigma\le\tau^*$ and $(y,u)=(x,v)$ on $[0,\sigma)$ a.s.. Moreover, with probability one, we have for each $p\ge 2$ and all $t\in [0,\tau^*)$,
  \begin{align}
    &\left|\frac{d\|x(t)\|_p}{dt}\right| \le \|v(t)\|_p, \label{est-2-1}\\
    &\|v(t)\|_p \le \|v(0)\|_p \exp\left(-\lambda\psi_* t -\frac{1}{2}\int_0^t D^2(s) ds + \int_0^t D(s)dW(s)\right). \label{est-2}
  \end{align}
\end{theorem}
\begin{proof}
 We will use a cutoff method and approximation process to construct the stopping time and strong solution. The proof will be separated in five steps.\newline

 \noindent $\bullet$ (Step 1). In this step, we will make a proper cutoff to the singular system \eqref{CS} and then we will obtain a sequence of global strong solutions to these cutoff regular systems. Let $\{a_n\}_{n\ge1}$ be a sequence of positive numbers satisfying $a_{n+1}<a_n$ for all $n\ge1$ and $a_n\to 0$ as $n\to\infty$. For each integer $n\ge 1$, we define
  \begin{equation*}
    \psi^n(r) = \psi(r)\ind_{[a_n,\infty)}(r) + \psi(a_n)\ind_{[0,a_n)}(r), \quad r\ge0.
  \end{equation*}
  Then obviously, each $\psi^n$ is locally Lipschitz from $[0,\infty)$ to itself. By Proposition \ref{regular}, the approximating system with $\psi^n$ in place of $\psi$ in \eqref{CS} and with the same initial data admits a unique global strong solution $(x^n,v^n)$ which is a continuous, $\{\F_t\}$-adapted, $\R^{2Nd}$-valued process.\newline

 \noindent $\bullet$ (Step 2). In this step, we will construct a sequence of stopping times $\{\tau_n\}_{n\ge1}$ to fulfill the conditions in Definition \ref{def-local}, via the solutions $(x^n,v^n)$ of cutoff systems. First, we define for each $n\ge 1$ a stopping time $\tau_n$ by
  \begin{equation*}
    \tau_{n} := \inf\left\{t\ge0: \min_{i\ne j}|x^n_i(t)-x^n_j(t)| \le a_n \right\}.
  \end{equation*}
  Then on $[0,\tau_n]$, we have $\min_{i\ne j}|x^n_i-x^n_j| \ge a_n$, and by the continuity of $x^n$,
  \begin{equation}\label{est-4}
    \min_{i\ne j}|x^n_i(\tau_{n})-x^n_j(\tau_{n})| = a_n \quad\text{on } \{\tau_{n}>0\}.
  \end{equation}
  As the decreasing property of the sequence $a_n$, we apply the definition of $\psi^n$ and have that $\psi^{n+1}(r)=\psi^n(r)$ when $r\in[a_n,\infty)$. On the other hand, according to Proposition \ref{regular}, the cutoff regular system admits a pathwise solution which is unique. Now, since the two solutions are governed by the same system before $\tau_n$ and they have same initial data, we can apply the uniqueness of solutions for the regular cutoff systems to conclude that,
  \[(x^{n+1},v^{n+1})=(x^n,v^n),\quad t\in[0,\tau_n]\ \text{a.s.}.\]
  Furthermore, by the fact that $a_{n+1}<a_n$ and the definition of $\tau_{n+1}$, we know that the minimal inter-particle distance of $x^{n+1}$ will not reach $a_{n+1}$ before $\tau_n$, which implies that $\tau_n\le\tau_{n+1}$ a.s.. For convenience, we properly redefine each $(x^n,v^n)$ on a common null set $\mathcal N\in\F$ such that $\tau_n\le\tau_{n+1}$ and $(x^{n+1},v^{n+1})=(x^n,v^n)$ for each $n\ge1$.\newline

\noindent $\bullet$ (Step 3). In this step, we will show that the limit of the sequence $\{\tau_n\}_{n\ge1}$ of stopping times is an a.s. \emph{positive} stopping time and prove. Since the sequence $\{\tau_n\}$ is non-decreasing, there limit is always well-define and is a stopping time \cite[Lemma 1.2.11]{KS91}. We denote, instead,
\begin{equation}\label{tau-star}
  \tau^* = \lim_{n\to\infty}\tau_n,
\end{equation}
for this limit has a special meaning as we will see in the later remark. Then we claim that $\tau^*>0$ a.s.. If $\tau^* = 0$ with positive probability, then on the event $\{\tau^* = 0\}$ we have for all $n\ge1$, $\tau_n=0$ and
  \begin{equation*}
    \min_{i\ne j}|x_i(0)-x_j(0)| = \min_{i\ne j}|x^n_i(0)-x^n_j(0)| \le a_n,
  \end{equation*}
  and hence $\min_{i\ne j}|x_i(0)-x_j(0)| = 0$. This contradicts with the assumption \eqref{initial-dist-no-coll} for $\mu_0$, since
  \begin{equation}\label{initial-non-coll}
    \begin{split}
      \P\left( \min_{i\ne j}|x_i(0)-x_j(0)| = 0 \right) &= \P\left( \exists i\ne j, x_i(0)=x_j(0) \right) \\
      &= \mu_0 \left( \exists i\ne j, x_i=x_j \right) = 0.
    \end{split}
  \end{equation}
  Thus we conclude that $\tau^*>0$ a.s.. As a consequence of \eqref{tau-star} and the monotonicity of $\{\tau_n\}$, we know that there exists $n_0\ge1$ such that for all $n\ge n_0$, $\P(\tau_n >0) >0$. \newline

\noindent $\bullet$ (Step 4). In this step, we will construct a strong solution up to the stopping time $\tau^*$ to the system \eqref{CS}, and prove the properties \eqref{collision-limit}--\eqref{est-2}. Define a continuous process $(x,v)$ up to $\tau^*$ by
  \begin{equation*}
    (x,v)(\omega)=(x^n,v^n)(\omega), \quad\text{on } [0,\tau_n(\omega)].
  \end{equation*}
  We will prove this is a strong solution in the sense of Definition \ref{def-local}. Firstly, it is easy to see that $(x,v)$ is a well-defined continuous process up to $\tau^*$, since $\tau_n$ is non-decreasing and $\tau^*$ is the limit of $\tau_n$. This implies Definition \ref{def-local}.(i).  Next, for all $t\ge0$ and $i=1,2,\cdots,n$, we have
  \begin{equation}\label{C4}
    \begin{split}
      x_i(t\wedge\tau_n) & = x^n_i(t\wedge\tau_n) = x_i(0) + \int_0^{t\wedge \tau_n} v^n_i(s) ds = x_i(0) + \int_0^t \ind_{\{s\le\tau_n\}}v^n_i(s) ds \\
         & = x_i(0) + \int_0^t \ind_{\{s\le\tau_n\}}v_i(s) ds = x_i(0) + \int_0^{t\wedge \tau_n} v_i(s) ds,
    \end{split}
  \end{equation}
  which provides the formulation of $x_i$ up to $\tau_n$. We can use the same method to obtain the formulation of $v_i$ up to $\tau^*$ as follows,
 \begin{equation}\label{C5}
   v_i(t\wedge\tau_n)  = v_i(0) + \frac{\lambda}{N}\sum_{j=1}^N \int_0^{t\wedge \tau_n} \psi(|x_i(s)-x_j(s)|)(v_j(s)-v_i(s))ds + \int_0^{t\wedge \tau_n}D(s) v_i(s)dW(s).
   \end{equation}
  Then, by \eqref{est-4} and the fact $a_n>0$, we have almost surely on the event $\{\tau_n>0\}$ that for all $i\neq j$ and $0\le t<\infty$,
  \begin{equation}\label{est-1}
    |x_i(t\wedge\tau_n)-x_j(t\wedge\tau_n)|=  |x_i^n(t\wedge\tau_n)-x_j^n(t\wedge\tau_n)|\geq \min_{k\ne m}|x^n_k(\tau_n)-x^n_m(\tau_n)| = a_n > 0,
  \end{equation}
  while on $\{\tau_n=0\}$, we have from \eqref{initial-non-coll} that almost surely \begin{equation}\label{est-5}
    |x_i(t\wedge\tau_n)-x_j(t\wedge\tau_n)|= |x_i(0)-x_j(0)|>0, \quad \forall i\neq j, 0\le t<\infty.
  \end{equation}
  Combining \eqref{C4}--\eqref{est-5}, we conclude that the equation in Definition \ref{def-local}.(iii) holds. Finally, on the event $\{\tau^*<\infty\}$, we have $\tau_n<\infty$ for all $n\ge1$. It follows from the definition of $\tau^*$ in \eqref{tau-star} and again \eqref{est-4} that a.s. on $\{\tau^*<\infty\}$,
  \begin{equation}\label{C7}
    \begin{aligned}
      &\ \liminf_{t\uparrow\tau^*} \min_{i\ne j}|x_i(t)-x_j(t)| = \lim_{n\rightarrow+\infty}\min_{i\ne j}|x_i(\tau_n)-x_j(\tau_n)| \\
      =&\ \lim_{n\rightarrow+\infty}\min_{i\ne j}|x^n_i(\tau_n)-x^n_j(\tau_n)| \\
      =&\ \lim_{n\rightarrow+\infty}\left( a_n\ind_{\{\tau_n>0\}} + \min_{i\ne j}|x_i(0)-x_j(0)| \ind_{\{\tau_n=0\}} \right) \\
      =&\ \lim_{n\rightarrow+\infty} a_n\ind_{\{\tau^*>0\}} + \min_{i\ne j}|x_i(0)-x_j(0)| \ind_{\{\tau^*=0\}} \\
      =&\ 0.
    \end{aligned}
  \end{equation}
We claim that \eqref{C7} implies that $\tau_n<\tau^*$ for all $n\ge1$ a.s. on $\{\tau^*<\infty\}$. Indeed, if $\P(\tau_m=\tau^*, \tau^*<\infty)>0$ for some $m\ge1$, then on the event $\{\tau_m=\tau^*, \tau^*<\infty\}$, we apply the continuity of $x$ and \eqref{C7} to obtain,
  \begin{equation}\label{C8}
    \min_{i\ne j}|x_i(\tau_m)-x_j(\tau_m)| = \liminf_{t\uparrow\tau_m} \min_{i\ne j}|x_i(t)-x_j(t)| = \liminf_{t\uparrow\tau^*} \min_{i\ne j}|x_i(t)-x_j(t)|=0.
  \end{equation}
  According to \eqref{est-1}, we know the distance between two particles at $\tau_m$ should be no less than $a_m$ in probability one, which implies that \eqref{C8} holds for probability zero. However, this contradicts to the assumption $\P(\tau_m=\tau^*, \tau^*<\infty)>0$. Hence we conclude that $\tau_n<\tau^*$ on the event $\{\tau^*<\infty\}$, which is Definition \ref{def-local}.(ii). Combining all above analysis, we have that $(x_i,v_i)$ is a strong solution in the sense of Definition \ref{def-local}.

Moreover, \eqref{C7} implies \eqref{collision-limit} on the event $\{\tau^*<\infty\}$. On the other hand, The estimate \eqref{est-2-1} is clear. We can apply \eqref{apriori-2} in Proposition \ref{regular} to have the following uniform bound for the cutoff system,
  \begin{equation*}
    \|v^n(t)\|_p \le \|v(0)\|_p \exp\left(-\lambda\psi_* t -\frac{1}{2}\int_0^t D^2(s) ds + \int_0^t D(s)dW(s) \right), \quad\forall t\ge0.
  \end{equation*}
As the RHS is independent of $n$ and $(x,v)=(x^n,v^n)$ before $\tau_n$, we conclude that $(x,v)$ has the estimate \eqref{est-2}. \newline

\noindent $\bullet$ (Step 5). In this step, we will prove by contradiction the uniqueness of the solution and the stoping time $\tau^*$. Suppose $(y,u)$ is another strong solution up to a stopping time $\sigma$. Then by the definition, there is a non-decreasing sequence of stopping times $\sigma_m\le\sigma$ such that $\lim_{m\to\infty} \sigma_m = \sigma$ and Definition \ref{def-local}.(iii) holds with $(y,u),\sigma_m$ in place of $(x,v),\tau_n$. We define a new non-decreasing sequence of stopping times by
  \begin{equation}\label{C9}
    \sigma'_m := \inf\left\{0\le t\le\sigma: \min_{i\ne j}|y_i(t)-y_j(t)| \le a_m \right\}\wedge\sigma_m.
  \end{equation}
  As $\sigma_m$ is non-decreasing, it is obvious that $\sigma'_m$ is also non-decreasing due to \eqref{C9}. Therefore, the sequence $\sigma'_m$ admits a limit and we set $\sigma':=\lim_{m\to\infty}\sigma'_m$.

  We claim that $\sigma' = \sigma$ with probability one. Indeed, we have  $\{\sigma' < \sigma\} = \cup_{k\ge1} \{\sigma'<\sigma_k\}$. And on each event $\{\sigma'<\sigma_k\}$ and for all $m\ge k$, we have $\sigma'_m\leq\sigma'<\sigma_k\leq\sigma_m$. Then, the definition of $\sigma'_m$ in \eqref{C9} implies that
  \begin{align}
 & \sigma'_m\leq\sigma'<\sigma_k\leq\sigma_m,\label{C10}\\
   &\sigma'_m := \inf\left\{0\le t\le\sigma: \min_{i\ne j}|y_i(t)-y_j(t)| \le a_m \right\}.\label{C11}
   \end{align}
  Therefore, we apply \eqref{C11} and the continuity of the strong solution to have $\min\limits_{i\neq j}|y_i(\sigma'_m)-y_j(\sigma'_m)|=a_m$ for all $m\ge k$. As $a_m$ decreases to zero, we apply the continuity of $y$ to obtain that
  \begin{equation}\label{C12}
    \min_{i\ne j}|y_i(\sigma')-y_j(\sigma')| = \lim_{m\to\infty} \min_{i\ne j}|y_i(\sigma'_m)-y_j(\sigma'_m)| =\lim_{m\to\infty}a_m= 0.
  \end{equation}
  On the other hand, according to Definition \ref{def-local}.(iii), we have $\min|y_i(t\wedge\sigma_m)- y_j(t\wedge\sigma_m)|>0$ with probability one for any $m\geq1$. Hence, we apply \eqref{C10}, \eqref{C12} and Definition \ref{def-local}.(iii) to obtain that the event $\{\sigma'<\sigma_k\}$ has probability zero, which implies that $\{\sigma'<\sigma\}$ has probability zero. Therefore, we finish the proof of the claim and conclude that $\sigma' = \sigma$ with probability one. Now note that
  \begin{equation*}
    \begin{split}
      y_i(t\wedge\sigma'_m) &= x^0_i + \int_0^{t\wedge \sigma'_m} u_i(s) ds, \\
      u_i(t\wedge\sigma'_m) & = v^0_i + \frac{\lambda}{N}\sum_{j=1}^N \int_0^{t\wedge \sigma'_m} \psi(|y_i(s)-y_j(s)|)(u_j(s)-u_i(s))ds +  \int_0^{t\wedge \sigma'_m} D(s) u_i(s)dW(s) \\
      & = v^0_i + \frac{\lambda}{N}\sum_{j=1}^N \int_0^{t\wedge \sigma'_m} \psi^m(|y_i(s)-y_j(s)|)(u_j(s)-u_i(s))ds + \int_0^{t\wedge \sigma'_m} D(s) u_i(s)dW(s).
    \end{split}
  \end{equation*}
  The pathwise uniqueness for the regular systems yields almost surely that $(y,u)=(x^m,v^m)=(x,v)$ on $[0,\sigma'_m]$ and $\sigma'_m\le \tau_m$ for all $m\ge1$.
  Therefore, $\sigma = \lim_{m\to\infty}\sigma'_m \le \lim_{m\to\infty} \tau_m= \tau^*$ and $(y,u)=(x,v)$ on $\cup_{m}[0,\sigma'_m] = [0,\sigma)$ a.s.. This finishes the proof.
\end{proof}

\begin{remark}\label{remark-2}
  (i). The uniqueness in Theorem \ref{local} also yields that the solution $(x,v)$ up to $\tau^*$ can not be extend to a solution up to any larger stopping time $\sigma$ (in the sense of Definition \ref{def-local}) with $\sigma\ge\tau^*$ and $\P(\sigma>\tau^*)>0$.

  (ii). From \eqref{collision-limit} and Remark \ref{remark-5}, we can see that the stopping time $\tau^*$ is nothing but the \emph{first collision time}, which is intuitive correctly since $\psi$ is not defined at $0$.

  (iii). By the forthcoming Lemma \ref{L4.2}, the statement \eqref{collision-limit} can actually be enhanced to
  \begin{equation*}
    \lim_{t\uparrow\tau^*} \min_{i\ne j}|x_i(t)-x_j(t)| = 0, \quad\text{on } \{\tau^*<\infty\}.
  \end{equation*}
  That is, there is no oscillation near the first collision time almost surely.

  (iv). A sufficient condition for the initial distribution $\mu_0$ to satisfy the collisionless assumption \eqref{initial-dist-no-coll} is that it has absolutely continuous first marginal distribution, namely, there exists a nonnegative function $\rho\in L^1(\R^{Nd})$, such that $\mu_0(A,\R^{Nd}) = \int_A\rho(x)dx$ for all $A\in\B(\R^{Nd})$. Indeed,
  \begin{equation*}
    \P\left( \exists i\ne j, x_i(0)=x_j(0) \right) \le \sum_{i\ne j} \P\left( x_i(0)=x_j(0) \right) = \sum_{i\ne j} \int_{x_i=x_j} \rho(x) dx = 0.
  \end{equation*}
  If we remove the collisionless assumption \eqref{initial-dist-no-coll} for $\mu_0$, then we can still obtain the existence of the unique strong solution up to a stopping time $\tau^*$, but only with a.s. $\tau^*\ge0$. In this case, the collisions may occur in the initial configuration.

  (v). If the noise intensity $D$ is identically zero, then for the non-collision initial data, the local existence is guaranteed before the first collision. Otherwise, suppose there are two agents stay together initially, one need to view them together as a larger particle, and the governed equations would be changed to CS model of $N-1$ agents with different mass, which is no longer the original symmetric CS model with $N$ agents. For this kind of models, please see discussions in \cite{HL09,MP18,ZZ20}.

  (vi). Theorem \ref{local} is still available for possible negative communication weight $\psi$ satisfying $\psi_*:=\inf_{r>0}\psi(r)>-\infty$.  However, since the flocking results in Section \ref{sec:5} as well as in \ref{App-A} and \ref{App-B} require $\psi_*\ge0$, we still set $\psi_*\ge0$ in force in this section and the next, in order to avoid the initial confusion.
\end{remark}

\section{Collision-avoidance for higher order singular systems}\label{sec:4}

In this section, we study under what conditions the system \eqref{CS} can avoid collisions and then have a global solution. According to \cite{CCM17,ZZ20}, in the case $\psi(r)=r^{-\alpha}$, $r>0$, the exponent $\alpha=1$ is critical for collision-avoidance for the original deterministic C-S model. We will prove that, with probability one, collisions will not occur in this case for the stochastic model \eqref{CS}. We first prove the collision-avoidance for one-dimensional case, which will clearly show our strategy.

\begin{proposition}\label{P4.1}
  Let $d=1$. Suppose the assumptions in Theorem \ref{local} hold. Assume in addition that
  \begin{itemize}
    \item[(i)] there exists $\e>0$ such that $\int_{\R^{2N}}\|v\|_2^\e\mu_0(dx,dv)<\infty$, and
    \item[(ii)] $\psi$ is non-increasing on $(0,\infty)$, and 
        $|\Psi(r)|\to+\infty$ as $r\to0$, where $\Psi$ is a primitive function of $\psi$.
  \end{itemize}
  Then the system \eqref{CS} admits a unique global strong solution. Moreover, with probability one the trajectories of this solution do not collide in finite time, i.e.,
  \begin{equation*}
    \P(x_i(t) \ne x_j(t), \forall 1\le i\ne j\le N, \forall t\ge0 ) = 1.
  \end{equation*}
\end{proposition}
\begin{remark}\label{remark-3}
  (i). Since any two primitive functions of $\psi$ are differed by a constant, the assumption that $|\Psi(r)|\to+\infty$ as $r\to0$ is independent of the choice of the primitive function $\Psi$. Moreover, it is easy to see from the singularity of $\Psi$ at 0 and the continuity and non-increasing of $\psi$ that
  \[\left\{ \begin{aligned}
    &\psi(r)\rightarrow +\infty, && \text{as } r\rightarrow 0,\\
    &\psi(r)\rightarrow \psi_*:=\inf_{r>0}\psi(r), && \text{as } r\rightarrow \infty.
  \end{aligned}\right.\]

  (ii). One can see from the assumptions that in our content, the collision-avoidance only relates to the singularity order of $\psi$ at 0 but has nothing to do with the decay rate of $\psi$ at infinity. We will see in Section \ref{sec:5} as well as in \ref{App-A} and \ref{App-B} where flocking behaviour is investigated, that the situation is exactly opposite, which means the flocking behaviour only relates to the decay rate of $\psi$ at infinity, but has nothing to do with the singularity order of $\psi$ at 0.

  (iii). In the case that $\psi(r) = r^{-\alpha} + \psi_*$, $r>0$, with some constant $\psi_*\ge0$, the primitive function is
  \begin{equation*}
    \Psi(r) =
    \begin{cases}
      \ln r + \psi_*t + C, & \mbox{if } \alpha=1, \\
      \frac{1}{1-\alpha}r^{1-\alpha} + \psi_*t + C, & \mbox{if } \alpha\ne1,
    \end{cases}
  \end{equation*}
  with some constant $C$. The assumptions on $\psi$ then reduce to $\alpha\ge1$. This example with $\psi_*=0$ is exactly the case studied in \cite{CCM17}. So basically, if we take $D\equiv 0$ and the deterministic initial data in Remark \ref{R2.6} which are required to be collisionless, then our result recovers the collision-avoidance results in \cite[Proposition 2.2, Theorem 2.1]{CCM17} very well.

  (iv). As we have said in Remark \ref{remark-2}.(vi), all results in this section are also available for the case $\psi_* >-\infty$. Intuitively, the negative value of $\psi$ only appears away from zero due to the assumption of $\psi$. Therefore, when two agents are closed to each other, the value of $\psi$ will be positive, and the mechanism to avoid collision will be same as positive communication case.

  (v). All these remarks are still available for the multi-dimensional case Theorem \ref{coll-avoid} by some obvious adaption in (i) and (iii), even though the assumption for $\psi$ will be slightly changed.
\end{remark}
\begin{proof}
  According to Theorem \ref{local}, there exists a unique stopping time $\tau^*$ such that either $\tau^*=+\infty$ or $\tau^*$ is finite, and that
  \[\liminf_{t\uparrow\tau^*} \min_{i\ne j}|x_i(t)-x_j(t)| = 0, \quad\text{a.s. on } \{\tau^*<\infty\}.\]
 In order to prove the avoidance of collision and the global existence of the strong solution, we only need to show $\P(\tau^* = \infty) = 1$, which is equivalently to that
 \begin{equation}\label{D1}
 \P(\tau^*\le T)=0,\quad \forall T>0.
 \end{equation}

\noindent   In the following, we will suppose that $\{\tau^*\le T\}$ is nonempty for some fixed $T>0$ and prove $\P(\tau^*\le T)=0$. As the proof is relatively long, we will split it into four steps.\newline

\noindent $\bullet$ (Step 1). In this step, we will construct the collision set and the corresponding differential equation of the particles in the collision set. As we already assume $\{\tau^*\le T\}$ is nonempty, we can define a random index set on $\{\tau^*\le T\}$ as follows,
  \begin{equation*}
    I(\omega) = \left\{ i\in\{1,2,\cdots,N\}: \exists j \ne i, \ \liminf_{t\uparrow\tau^*(\omega)} |x_i(t)(\omega)-x_j(t)(\omega)| = 0 \right\}.
  \end{equation*}
  The set $I$ can be regarded as the collision set which contains all collision particles at the stopping time $\tau^*$. By \eqref{collision-limit}, we know that $I\ne\emptyset$ on $\{\tau^*\le T\}$ a.s.. Then, we can further define a random index $l:\{\tau^*\le T\}\to\{1,2,\cdots,N\}$ as follows,
  \begin{equation}\label{D3}
    l(\omega) = \arg\max\{x_i(0)(\omega): i\in I(\omega)\}.
  \end{equation}
  The random variable $l$ is well-defined except on on a null subset of $\Omega$, since by \eqref{initial-non-coll} $x_i(0)\ne x_j(0)$ for all $i\ne j$ a.s.. According to \eqref{D3}, for almost all $\omega\in\{\tau^*\le T\}$, we have $x_i(0)(\omega) < x_l(0)(\omega)$ for all $i\in I(\omega)$ and $i\ne l(\omega)$. Moreover, by Definition \ref{def-local}.(iii), $x_i(t) \ne x_j(t)$ for all $i\ne j$ and $t\in[0,\tau^*)$. Therefore, we have $x_i(t)<x_l(t)$ for all $i\in I, i\ne l$ over the whole time interval $[0,\tau^*)$. Now we define
  \begin{equation*}
    \xi(t) = \sum_{i\in I} (v_l(t) - v_i(t)), \quad t\in[0,\tau^*).
  \end{equation*}
  As $(x_i,v_i)$ is the strong solution to the system \eqref{CS}, we can differentiate $\xi$ and obtain the stochastic differential of $\xi$ over $[0,\tau^*)$ as below,
  \begin{equation}\label{D4}
    \begin{split}
      d\xi =&\ \frac{\lambda}{N} \sum_{i\in I} \sum_{j=1}^N \psi(|x_l-x_j|)(v_j-v_l)dt - \frac{\lambda}{N} \sum_{i\in I} \sum_{j=1}^N \psi(|x_i-x_j|)(v_j-v_i)dt\\
      &\ + D \sum_{i\in I} (v_l - v_i) dW \\
      =&\ \frac{\lambda}{N} \sum_{i,j\in I} \psi(|x_l-x_j|)(v_j-v_l)dt - \underbrace{\frac{\lambda}{N} \sum_{i,j\in I} \psi(|x_i-x_j|)(v_j-v_i)}_{=0}dt \\
      &\ + \underbrace{\frac{\lambda}{N} \sum_{i\in I} \sum_{j\notin I} \left[ \psi(|x_l-x_j|)(v_j-v_l) - \psi(|x_i-x_j|)(v_j-v_i) \right]}_{=:\Xi} dt + D \sum_{i\in I} (v_l - v_i) dW \\
      =&\ \frac{\lambda |I|}{N}\sum_{j\in I, j\ne l} \psi(|x_l-x_j|)(v_j-v_l)dt + \Xi dt + D \xi dW,
    \end{split}
  \end{equation}
  where we apply the antisymmetry in the indexes $i$ and $j$ to obtain the cancellation.\newline

 \noindent $\bullet$ (Step 2). In this step, we will make use of the equation \eqref{D4} to construct a Lyapunov functional. Actually, we integrate both sides of \eqref{D4} over $[0,t\wedge \sigma]$ for any stopping time $\sigma<\tau^*$ to obtain
  \begin{equation}\label{D5}
    \begin{split}
      \xi(t\wedge\sigma) - \xi(0) =&\ -\frac{\lambda |I|}{N}\sum_{j\in I, j\ne l} \Psi(x_l(t\wedge\sigma)-x_j(t\wedge\sigma)) + \frac{\lambda |I|}{N}\sum_{j\in I, j\ne l} \Psi(x_l(0)-x_j(0)) \\
      &\ + \int_0^{t\wedge\sigma} \Xi(s) ds + \int_0^{t\wedge\sigma}D(s) \xi(s) dW(s).
    \end{split}
  \end{equation}
Now we use the function $\Psi$ as a Lyapunov functional (cf. \cite{HL09}) and do some estimates for it. According to the definition of the collision set $I$, each $x_j$, $j\notin I$ will not collide with any other particles. Therefore, there exists a nonnegative random variable $\delta$ which is positive on $\{\tau^*\le T\}$ such that a.s. on $\{\tau^*\le T\}$,
  \begin{equation}\label{D7}
    |x_i-x_j|\ge \delta>0, \quad \text{uniformly on } [0,\tau^*),\quad \forall i\in I, j\notin I.
  \end{equation}
  Indeed, we can set $\delta:=\inf_{t\in[0,\tau^*)} |x_i(t)-x_j(t)|$ for $i\in l$ and $j\notin l$. Then, for the term $\Xi$ in \eqref{D4}, we use \eqref{est-2} and \eqref{D7} as well as the assumption that $\psi$ is non-increasing, to have
  \begin{equation}\label{D8}
    \begin{split}
      \int_0^{t\wedge\sigma} |\Xi(s)| ds & \le \frac{4\lambda |I|(N-|I|)}{N} (\psi(\delta\wedge r_0) \vee |\psi_*|) \int_0^{t\wedge\sigma} \|v(s)\|_2 ds \\
         & \le 4\lambda N (\psi(\delta\wedge r_0) \vee |\psi_*|) \|v(0)\|_2 \int_0^{t\wedge\sigma} \exp\left(-\frac{1}{2}\int_0^s D^2(u) du + \int_0^s D(u)dW(u) \right) ds,
    \end{split}
  \end{equation}
  where $r_0>0$ is the first zero of $\psi$, that is,
  \begin{equation*}
    r_0 := \inf\{r>0: \psi(r) = 0 \},
  \end{equation*}
  and we adopt the convention $\inf\emptyset = +\infty$ here, so that if $\psi$ has no zeros, or equivalently, $\psi>0$, then $r_0 = +\infty$.
  On the other hand, we define a sequence of stopping time as below,
  \begin{equation*}
    \sigma_n := \inf\left\{t\ge0: \min_{j\in I, j\ne l}|x_j(t)-x_l(t)| \le a_n \right\}, \quad n\ge1,
  \end{equation*}
  where $\{a_n\}_{n\ge1}$ is a decreasing sequence of positive numbers that converges to $0$ as in the proof of Theorem \ref{local}. Then it is easy to see that on $\{\tau^*\le T\}$, almost surely, $\sigma_n<\tau^*$ and
  \begin{equation}\label{est-6}
    \lim_{n\to\infty}\min_{j\in I, j\ne l}|x_l(\sigma_n)-x_j(\sigma_n)| = 0.
  \end{equation}
  Now on the event $\{\tau^*\le T\}$, we combine \eqref{est-2}, \eqref{D5}, \eqref{D8} and the definition of $\sigma_n$ to have the following estimate for the Lyapunov functional $\Psi$,
  \begin{equation}\label{D10}
    \begin{aligned}
      &\ \frac{\lambda |I|}{N} \left| \sum_{j\in I, j\ne l} \Psi(|x_l(\sigma_n)-x_j(\sigma_n)|) \right| \\
      \le&\ \frac{\lambda |I|}{N}\sum_{j\in I, j\ne l} \left| \Psi(|x_l(0)-x_j(0)|) \right| + |\xi(\sigma_n)| + |\xi(0)| + \int_0^{\sigma_n} |\Xi(s)| ds + \left|\int_0^{\sigma_n} D(s)\xi(s) dW(s)\right| \\
      \le&\ \lambda(N-1) \max_{i\ne j} \left|\Psi\left(|x_i(0)-x_j(0)|\right) \right| + 4(N-1)\|v(0)\|_2 \sup_{0\le t\le T} \exp\left(-\frac{1}{2}\int_0^t D^2(s) ds + \int_0^t D(s)dW(s)\right) \\
      &\ + 4\lambda N (\psi(\delta\wedge r_0) \vee |\psi_*|) \|v(0)\|_2 \int_0^{T} \exp\left(-\frac{1}{2}\int_0^s D^2(u) du + \int_0^s D(u)dW(u)\right) ds \\
      &\ + \sup_{0\le t\le T}\left|\int_0^{t\wedge\sigma_n} D(s)\xi(s) dW(s)\right| \\
      =: &\ J_1+J_2+J_3+J^n_4.
    \end{aligned}
  \end{equation}
  Note that only the fourth term $J^n_4$ depends on $n$ on the RHS.\newline

\noindent $\bullet$ (Step 3). In this step, we will estimate $J_i$'s in \eqref{D10} term by term, and then induce a contradiction. From \eqref{est-2-1} and \eqref{est-2} we know that on $\{\tau^*\le T\}$, for all $j\in I$,
\begin{equation}\label{est-21}
  \begin{split}
    &\ |x_l(\sigma_n)-x_j(\sigma_n)| \le 2\|x(\sigma_n)\|_2 \le \int_0^{\sigma_n} \|v(t)\|_2dt \\
    \le&\ \|v(0)\|_p \int_0^T \exp\left( -\frac{1}{2}\int_0^t D^2(s) ds + \int_0^t D(s)dW(s)\right) dt < \infty.
  \end{split}
\end{equation}
Since $\psi$ is locally bounded on $(0,\infty)$, its primitive $\Psi$ is locally Lipschitz and hence also locally bounded on $(0,\infty)$. This together with \eqref{est-6}, \eqref{est-21} and the assumption $|\Psi(r)|\to\infty$ as $r\to0$ implies that
\begin{equation*}
  \{\tau^*\le T\} \subset \left\{ \lim_{n\to\infty}\left| \sum_{j\in I, j\ne l} \Psi(|x_l(\sigma_n)-x_j(\sigma_n)|)\right| = \infty \right\}.
\end{equation*}
We apply this to \eqref{D10} and have
  \begin{equation}\label{D11}
    \begin{split}
      &\ \P(\tau^*\le T)\\
      =&\ \P\left( \tau^*\le T, \lim_{n\to\infty}\left| \sum_{j\in I, j\ne l} \Psi(|x_l(\sigma_n)-x_j(\sigma_n)|)\right| = \infty\right) \\
      \le&\ \P(J_1=\infty) + \P(J_2=\infty) + \P(\tau^*\le T, J_3=\infty) + \P\left(\liminf_{n\to\infty} J^n_4=\infty\right) \\
      =&\ \P(J_1=\infty) + \lim_{K\to\infty}\P(J_2 \ge K) + \lim_{K\to\infty}\P( \tau^*\le T, J_3 \ge K) \\
      &\ + \lim_{K\to\infty}\P\left(\liminf_{n\to\infty} J^n_4\ge K \right).
    \end{split}
  \end{equation}
  In order to estimate the term $J_1$, we split in two situations. Note that $\Psi$ is locally bounded on $(0,\infty)$. 
  In the case that $\limsup_{r\to\infty}|\Psi(r)| < \infty$, $\Psi$ is bounded on $[r,\infty)$ for all $r>0$. Then the assumption $|\Psi(r)|\to+\infty$ as $r\to0$ yields
  \begin{equation}\label{D13}
    \begin{split}
      \P(J_1=\infty) &= \P\left( \max_{i\ne j} \left|\Psi\left(|x_i(0)-x_j(0)|\right) \right| = \infty \right) \\
      &= \P\left( \exists i\ne j, |x_i(0)-x_j(0)| = 0 \right) \\
      &= \P\left( \min_{i\ne j}|x_i(0)-x_j(0)| = 0 \right) \\
      &= 0,
    \end{split}
  \end{equation}
  where the last equality uses the collisionlessness of $\mu_0$ as in \eqref{initial-non-coll}.
  In the case that $\limsup_{r\to\infty}|\Psi(r)| = \infty$, $\Psi$ has singularity at both 0 and $\infty$. Hence,
  \begin{equation}\label{est-16}
    \begin{split}
      \P(J_1=\infty) &= \P\left( \max_{i\ne j} \left|\Psi\left(|x_i(0)-x_j(0)|\right) \right| = \infty \right) \\
      &= \P\left( \exists i\ne j, |x_i(0)-x_j(0)| = 0 \text{ or } |x_i(0)-x_j(0)| = \infty\right) \\
      &\le \P\left( \min_{i\ne j}|x_i(0)-x_j(0)| = 0 \right) + \P\left( \|x(0)\|_2= \infty\right) \\
      &= 0,
    \end{split}
  \end{equation}
  where the second term at RHS of the third inequality vanishes since the distribution of $\|x(0)\|_2$ is inner regular (or tight) as a probability measure on $(\R,\B(\R))$.
  For the term $J_2$, we first split in two parts, which will be much clearer and simpler to estimate. In fact, we have
  \begin{equation*}
    \begin{split}
      \P(J_2 = \infty) \le \P\left( \|v(0)\|_2 =\infty \right) + \lim_{K\to\infty} \P\left( \sup_{0\le t\le T} \exp\left(-\frac{1}{2}\int_0^t D^2(s) ds + \int_0^t D(s)dW(s) \right) \ge K \right).
       \end{split}
  \end{equation*}
   Then, we apply the regular property of the distribution of $\|v(0)\|_2$ to obtain the vanishing of the first term of the RHS in the above inequality. Moreover, we use the Doob's maximal inequality to deal with the second term above. Actually, we have
  \begin{equation}\label{D14}
    \begin{split}
      \P(J_2 = \infty)& \le 0 + \lim_{K\to\infty}\P\left( \sup_{0\le t\le T} e^{\int_0^t D(s)dW(s) } \ge K \right) \\
      &\le \lim_{K\to\infty}  \frac{1}{K}\E\left( e^{\int_0^T D(s)dW(s) }\right) \le \lim_{K\to\infty} \frac{1}{K}e^{\frac{1}{2}\int_0^T D^2(s) ds} = 0.
    \end{split}
  \end{equation}
 For the term $J_3$, similarly as what we did for $J_2$, we can first split $J_3$ into three parts as follows,
 \begin{equation*}
     \begin{split}
      \P(\tau^*\le T, J_3 = \infty) \le&\ \P\left( \tau^*\le T, \psi(\delta\wedge r_0) \vee |\psi_*| =\infty \right) + \P\left( \|v(0)\|_2 = \infty \right)  \\
      & + \lim_{K\to\infty}\P\left( \int_0^{T} \exp\left(-\frac{1}{2}\int_0^s D^2(u) du + \int_0^s D(u)dW(u) \right) ds \ge K \right).
         \end{split}
  \end{equation*}
  For the first term at RHS, since $|\psi_*|<\infty$ and $\psi$ only has singularity at 0 as explained in Remark \ref{remark-3}.(i), 
  it follows that
  \begin{equation*}
    \P\left( \tau^*\le T, \psi(\delta\wedge r_0) \vee |\psi_*| =\infty \right) 
    = \P\left( \tau^*\le T, \delta = 0 \right) 
    = 0,
  \end{equation*}
  due to the definition 
  of $\delta$. Similar as before, the inner regularity of the distribution of $\|v(0)\|_2$ yields the vanishing of $\P( \|v(0)\|_2 = \infty )$. Finally, we apply Markov's inequality and Fubini's theorem to estimate
  \begin{equation}\label{est-18}
    \begin{split}
      &\ \P(\tau^*\le T, J_3 = \infty)  \\
      \le &\ 0 + 0 + \lim_{K\to\infty} K^{-1} \int_0^{T} \E \exp\left(-\frac{1}{2}\int_0^s D^2(u) du + \int_0^s D(u)dW(u)\right) ds \\
      =&\ \lim_{K\to\infty} K^{-1} T \\
      =&\ 0.
    \end{split}
  \end{equation}
Now, for the fourth term $J^n_4$ and any $p>0$, we use Markov's inequality  and Fatou's lemma to obtain the following estimate,
 \begin{equation*}
    \begin{split}
      \P\left(\liminf_{n\to\infty} J^n_4\ge K \right)  \le \frac{1}{K^{2p}} \liminf_{n\to\infty} \E\left[ \left( \sup_{0\le t\le T}\left|\int_0^{t\wedge\sigma_n} D(s)\xi(s) dW(s)\right| \right)^{2p} \right].
     \end{split}
  \end{equation*}
Next, for any $q>1$, we apply the Burkholder-Davis-Gundy (BDG) inequality \cite[Theorem 3.3.28]{KS91}, H\"older's inequality and \eqref{est-2} to above inequality and obtain
  \begin{align}
      &\ \P\left(\liminf_{n\to\infty} J^n_4\ge K \right)\label{est-19} \\
      \le&\ \left(\frac{D^*(T)}{K}\right)^{2p} \liminf_{n\to\infty} \E\left[\left( \int_0^{T\wedge\sigma_n} |\xi(s)|^2 ds \right)^{p} \right] \notag\\
      \le&\  \left(\frac{2(N-1)D^*(T)}{K}\right)^{2p}\E\left[\|v(0)\|_2^{2p}\left(\int_0^{T} \exp\left(-\int_0^s D^2(u) du + 2\int_0^s D(u)dW(u) \right) ds \right)^{p}\right] \notag\\
      \le&\ \left(\frac{2(N-1)D^*(T)}{K}\right)^{2p} \left[\E\left(\|v(0)\|_2^{2pq}\right)\right]^{1/q} \notag\\
      &\ \times\left( \E\left[\left(\int_0^{T} \exp\left(-\int_0^s D^2(u) du + 2\int_0^s D(u)dW(u)  \right) ds \right)^{pq'}\right] \right)^{1/q'},\notag
  \end{align}
 where we let $D^*(T)=\sup_{0\leq t\leq T}D(t)$ and $q'$ is the conjugate of $q$, i.e., $\frac{1}{q}+\frac{1}{q'}=1$. Choosing $q\ge \frac{\e}{2}+1$ and $p=\frac{\e}{2q}$, so we have $pq' \le1$. Now we apply Jensen's inequality and Fubini's theorem to have
  \begin{equation}\label{D17}
    \begin{split}
      &\ \E\left[\left(\int_0^{T} \exp\left(-\int_0^s D^2(u) du + 2\int_0^s D(u)dW(u)  \right) ds \right)^{pq'}\right] \\
      \le&\ \left(\int_0^{T} \E \exp\left(-\int_0^s D^2(u) du + 2\int_0^s D(u)dW(u) \right) ds \right)^{pq'} \leq \left(\int_0^{T} e^{\left(D^*(T)\right)^2 s} ds \right)^{pq'}.
    \end{split}
  \end{equation}
Note, as defined in the statement of the proposition, $\e$ has the property that $\int_{\R^{2N}}\|v\|_2^\e\mu_0(dx,dv)<\infty$. Therefore, we combine \eqref{est-19} and \eqref{D17} to obtain
  \begin{equation}\label{D18}
    \lim_{K\to\infty} \P\left(\liminf_{n\to\infty} J^n_4\ge K \right) \le \lim_{K\to\infty} c(N,T,D^*(T),\e,q) \left( \int_{\R^{2N}}\|v\|_2^\e\mu_0(dx,dv) \right)^{1/q} K^{-2p} = 0.
  \end{equation}
 Finally, we combine \eqref{D11}, \eqref{est-16}, \eqref{D13}, \eqref{D14}, \eqref{est-18} and \eqref{D18} to conclude that $\P(\tau^*\le T)=0$. By the arbitrary choice of $T$, we verify \eqref{D1} and thus finish the proof.
\end{proof}
In the proof of the Proposition \ref{P4.1}, the key step is to find a proper Lyapunov functional. In one-dimensional case, this can be done by the anti-derivative. However, in multi-dimensional case, the anti-derivative cannot be directly applied to the system. Therefore, we need to construct a new Lyapunov functional and use more delicate estimates to yield desired results.

Before the main theorem in the section, we will present an a priori lemma which shows that almost surely there is  no oscillation near the first collision time. This lemma actually holds in the general setting of Section \ref{sec:3}, as indicated by the proof.
\begin{lemma}\label{L4.2}
  Suppose the assumptions in Theorem \ref{local} hold. Let $(x,v)$ be a strong solution to the system \eqref{CS} up to the unique stopping time $\tau^*>0$. Then for all $i,j\in\{1,2,\cdots,N\}$, on the event $\{\tau^*<\infty, \liminf\limits_{t\uparrow\tau^*} |x_i(t)-x_j(t)| = 0\}$, it holds a.s. that $\lim\limits_{t\uparrow\tau^*} |x_i(t)-x_j(t)| = 0$.
\end{lemma}

\begin{proof}
  For any positive real numbers $a$ and $b$ such that $0<a<b$, we define two sequences $\{\sigma_n\}$ and $\{\zeta_n\}$ of stopping times inductively as follows,
  \begin{equation}\label{D19}
  \begin{aligned}
    &\sigma_1 := \inf\{0\le t < \tau^*: |x_i(t)-x_j(t)|>b\}, \\
    &\zeta_n := \inf\{\sigma_n< t < \tau^*: |x_i(t)-x_j(t)|<a\}, \\
    &\sigma_{n+1} := \inf\{\zeta_n< t < \tau^*: |x_i(t)-x_j(t)|>b\},
  \end{aligned}
  \end{equation}
  where we employ the standard convention that the infimum of the empty set is infinity. Thus, we have either $\sigma_n=\infty$ or $\sigma_n<(\tau^*\wedge\infty)$, and so does $\zeta_n$. According to Definition \ref{def-local}, we have $|x_i-x_j|\ne0$ on $[0,\tau^*)$ a.s.. Therefore, we can always find a small enough positive number $K>0$ such that
  \begin{equation}\label{D20}
  \inf\{0\le t < \tau^*: |x_i(t)-x_j(t)|\ge K\}<\infty,\quad \text{a.s.}.
  \end{equation}
Then we let $b<K$. According to \eqref{D20}, the choice of $b$ and the definition of $\sigma_1$ in \eqref{D19}, we obtain that
\[\sigma_1<(\tau^*\wedge\infty),\quad  \text{a.s.}.\]
\noindent $\bullet$ Claim:  If $\tau^*<\infty$ and $\liminf_{t\uparrow\tau^*} |x_i(t)-x_j(t)| = 0$, then there exists $n\ge1$ such that $\zeta_n<\tau^*$ and $\sigma_{n+1}=\infty$ with probability one.

Before the proof of the claim, we would like to explain the motivation. Actually, the claim shows that, if $\tau^*<\infty$ and $\liminf_{t\uparrow\tau^*} |x_i(t)-x_j(t)| = 0$, then there exists $n\ge1$ such that $\zeta_n<\tau^*$ and $|x_i(t)-x_j(t)|\le b$ for all $t\in (\zeta_n,\tau^*)$ with probability one. As $b\in(0,K)$ can be chosen arbitrarily small, we obtain the desired result $\lim_{t\uparrow\tau^*} |x_i(t)-x_j(t)| = 0$ a.s.. Therefore, we only need to verify the claim in the rest of the proof.\newline

\noindent Proof of the claim: The claim is equivalent to that, $\tau^*=\infty$ or $\liminf_{t\uparrow\tau^*} |x_i(t)-x_j(t)| > 0$ a.s. on the event $\{\exists n\ge1, \text{s.t. } \sigma_n<\tau^* \text{ and } \zeta_n=\infty\}\cup \{\forall n\ge1, \sigma_n<\tau^*\} =: A\cup B$. On the event $A$, It is obvious that $\liminf_{t\uparrow\tau^*} |x_i(t)-x_j(t)| \ge a > 0$ holds due to \eqref{D19}. Thus, we only need to prove $\tau^*=\infty$ holds a.s. on the event $B$, which is equivalent to show that for any $T>0$,
  \begin{equation}\label{D21}
    \P(\tau^*\le T \text{ and } \forall n\ge1, \sigma_n<\tau^*) = 0.
  \end{equation}
 According to the definition of $\sigma_n$ and $\zeta_n$, it is obviously that $\sum_{n=1}^\infty (\zeta_n-\sigma_n) < \infty$ a.s. on the event $\{\tau^*\le T\} \cap \{\forall n\ge1, \sigma_n<\tau^*\}$. Therefore, in order to prove \eqref{D21}, we only need to prove that $\sum_{n=1}^\infty (\zeta_n-\sigma_n) = \infty$ also holds a.s. on the event $\{\tau^*\le T\} \cap \{\forall n\ge1, \sigma_n<\tau^*\}$. More precisely, we need to prove that
  \begin{equation*}
    \prod_{n=1}^\infty \left( \ind_{\{\sigma_n<\tau^*\le T\}} e^{-(\zeta_n-\sigma_n)} \right) = \left( \prod_{n=1}^\infty \ind_{\{\sigma_n<\tau^*\le T\}} \right) \exp\left( - \sum_{n=1}^\infty (\zeta_n-\sigma_n)\right) = 0, \quad\text{a.s.,}
  \end{equation*}
which is obviously equivalent to show that
  \begin{equation}\label{est-12}
 Z_n:=\ind_{\{\sigma_n<\tau^*\le T\}} e^{-(\zeta_n-\sigma_n)},\quad   \E \left[ \prod_{n=1}^\infty Z_n\right] = 0.
  \end{equation}
 In the following, we will prove \eqref{est-12} and thus finish the proof of the claim. According to the Markov property, we have for $m\ge 1$ that,
  \begin{equation}\label{est-11}
       \ \E \left[ \prod_{n=1}^{m+1} Z_n \Bigg| \F_{\sigma_{m+1}} \right] =\ \ind_{\{\sigma_{m+1}<\tau^*\le T\}} \E \left( e^{-(\zeta_{m+1}-\sigma_{m+1})} \Big| \F_{\sigma_{m+1}} \right) \prod_{n=1}^{m} Z_n.
  \end{equation}
  Next, we define $d_{ij}:=|x_i-x_j|$. 
  Then similar to \eqref{est-20}, we have
  \begin{equation*}
    |d_{ij}'| = \frac{|\langle x_i- x_j, v_i-v_j\rangle|}{|x_i- x_j|} \le |v_i-v_j| \le 2\|v\|_2.
  \end{equation*}
  Thus for $t\in [0,\tau^*-\sigma_{m+1})$, we 
  have the following estimates,
   \begin{equation*}
    \begin{split}
      |d_{ij}(t+\sigma_{m+1})- d_{ij}(\sigma_{m+1})| &\le 2 \int_{\sigma_{m+1}}^{t+\sigma_{m+1}} \|v(s)\|_2ds = 2 \int_0^t \|v(s+\sigma_{m+1})\|_2ds.
    \end{split}
  \end{equation*}
  Hence, 
  \begin{equation}\label{est-7}
    \begin{split}
      &\ \left\{ |d_{ij}(t+\sigma_{m+1})- d_{ij}(\sigma_{m+1})|> b-a \right\} \\
      \subset&\ \left\{ 2\int_0^t \|v(s+\sigma_{m+1})\|_2 ds > b-a \right\},
    \end{split}
  \end{equation}
  For each real number $c>0$, we define a stopping time
  $$
  \eta_c:= \inf\left\{t\ge0: 
  2\int_0^t \|v(s+\sigma_{m+1})\|_2 ds> c\right\}.
  $$
  Obviously, the process 
  $\int_0^\cdot \|v(s+\sigma_{m+1})\|_2 ds$ is a continuous increasing process on $[0,\tau^*-\sigma_{m+1})$, and it starts from the origin since $\|v\|_2$ is almost surely finite due to \eqref{est-2}. Therefore, we have $\eta_c>0$ a.s. for each $c>0$. Then, according to \eqref{est-7}, we have the following estimates,
  \begin{equation}\label{D25}
    \inf\{t\ge0: |d_{ij}(t+\sigma_{m+1})- d_{ij}(\sigma_{m+1})| > b-a \} \ge 
    \eta_{(b-a)} > 0, \quad \text{a.s.}.
  \end{equation}
  According to the definition of $\sigma_n$ and $\zeta_n$ in \eqref{D19} and the first inequality in \eqref{D25}, we obtain
  \begin{equation}\label{est-23}
    \ind_{\{\sigma_{m+1}<\tau^*\le T\}} \E \left( e^{-(\zeta_{m+1}-\sigma_{m+1})} \Big| \F_{\sigma_{m+1}} \right) \le \E \left[ \exp\left( - 
    \eta_{(b-a)} \right) \right] =: \e_0.
  \end{equation}
  It follows from the second inequality in \eqref{D25} that $\e_0<1$. 
  Thus, \eqref{est-23} together with \eqref{est-11} yields that for all $m\ge1$,
 \begin{equation}\label{D26}
 \begin{aligned}
   &\ \E \left[ \prod_{n=1}^{m+1}Z_n \right] = \E\left[\E \left[ \prod_{n=1}^{m+1} Z_n \Bigg| \F_{\sigma_{m+1}} \right] \right]\\
   =&\ \E\left[\ind_{\{\sigma_{m+1}<\tau^*\le T\}} \E \left( e^{-(\zeta_{m+1}-\sigma_{m+1})} \Big| \F_{\sigma_{m+1}} \right) \prod_{n=1}^{m} Z_n\right]\leq \e_0\E\left[ \prod_{n=1}^{m} Z_n\right].
 \end{aligned}
 \end{equation}
 According to \eqref{D19} and the definition of $Z_n$ in \eqref{est-12}, we have $Z_n\leq 1$. Therefore, we apply \eqref{D26} and inductive criteria to obtain  that
  \begin{equation*}
    \E \left[ \prod_{n=1}^\infty Z_n \right] \le \E \left[ \prod_{n=1}^{m+1} Z_n \right] \le \e_0^{m+1},\quad \forall m\geq 1.
  \end{equation*}
  By the fact that $\e_0<1$ and the arbitrary choice of $m$, we conclude that  \eqref{est-12} holds, and thus finish the proof of the claim.
\end{proof}

Now we are ready to prove the main theorem in this section, which extends the result in Proposition \ref{P4.1} and shows the collision-avoidance in multi-dimensional case, which naturally implies the global existence of the strong solution. It is worth to note that the assumption on $\psi$ is slightly different from that of Proposition \ref{P4.1}.
\begin{theorem}[Collision-avoidance and global existence]\label{coll-avoid}
  Let $d\ge1$. Suppose the assumptions in Theorem \ref{local} hold. Assume in addition that
  \begin{itemize}
\item[(i)] there exists $\e>0$ such that $\int_{\R^{2N}}\|v\|_2^\e\mu_0(dx,dv)<\infty$, and
\item[(ii)] $\psi$ is global Lipschitz on $[r,\infty)$ for every $r>0$, and 
    $|\Psi_*(r)|\to+\infty$ as $r\to0$, where $\Psi_*$ is a primitive function of the infimum function $\psi_*(r) := \inf_{0< s\le r}\psi(s)$.
\end{itemize}
  Then the system \eqref{CS} admits a unique global strong solution. Moreover, with probability one the trajectories of this solution do not collide in finite time, i.e.,
  \begin{equation}\label{D27}
    \P(x_i(t) \ne x_j(t), \forall 1\le i\ne j\le N, \forall t\ge0 ) = 1.
  \end{equation}
\end{theorem}

\begin{proof}
\noindent $\bullet$ (Step 1). In this step, we will introduce a family of collision sets and find an equivalent statement of \eqref{D27}. Let $\tau^*$  be the unique stopping time defined in Theorem \ref{local}. Same as Proposition \ref{P4.1}, in order to prove \eqref{D27}, we only need to prove  $\P(\tau^*\le T)=0$ for all $T>0$. For this purpose, we define for each $l\in\{1,2,\cdots,N\}$ a random index set as below,
  \begin{equation*}
    I_l := \left\{ i\in\{1,2,\cdots,N\}:\  \liminf_{t\uparrow\tau^*} |x_i(t)-x_l(t)| = 0 \right\}.
  \end{equation*}
We use $|I_l|$ to denote the cardinality of the random set $I_l$. Clearly, each $|I_l|$ is a random variable valued in $\{1,2,\cdots,N\}$. According to \eqref{collision-limit}, on the event $\{\tau^*<\infty\}$, there exist $i,j\in\{1,2,\cdots,N\}$ such that $i\ne j$ and $\liminf_{t\uparrow\tau^*} |x_i(t)-x_j(t)| = 0$ a.s.. Then we obtain that  $\cup_{l=1}^N (I_l\setminus \{l\})$ is nonempty a.s. on $\{\tau^*<\infty\}$. Therefore, to prove \eqref{D27} is equivalent to show that,
\begin{equation}\label{D29}
\P(\tau^*\le T, |I_l| >1 )=0,\quad  \forall T>0, \quad \forall l\in\{1,2,\cdots,N\}.
\end{equation}
From now on, we will fix a $T>0$ and an $l\in\{1,2,\cdots,N\}$, and prove $\P(\tau^*\le T, |I_l| >1 )=0$.

Furthermore, according to Lemma \ref{L4.2}, there will be no oscillation near the stopping time $\tau^*$ on $\{\tau^*<\infty\}$ a.s.. It follows that
  \begin{equation*}
    I_l := \left\{ i\in\{1,2,\cdots,N\}: \lim_{t\uparrow\tau^*} |x_i(t)-x_l(t)| = 0 \right\}, \quad\text{a.s. on } \{\tau^*<\infty\}.
  \end{equation*}
  Hence, similar as \eqref{D7}, there exists a nonnegative random variable $\delta$ which is positive on $\{\tau^*\le T\}$, such that the following statements hold a.s. on $\{\tau^*\le T, |I_l| >1\}$,
  \begin{align}
    \text{for } &\  i,j \in I_l, \ i\ne j,\  \lim_{t\uparrow\tau^*} |x_i(t)-x_j(t)| = 0, \label{est-13}\\
    \text{for } &\  i\in I_l,\  k\notin I_l, \ |x_i-x_k|\ge \delta>0\  \text{ uniformly on } [0,\tau^*). \label{est-14}
  \end{align}

 \noindent $\bullet$ (Step 2). In this step, we will study the dynamics of the system \eqref{CS}, and follow the idea in Proposition \ref{P4.1} to construct a proper Lyapunov functional. For simplicity, we introduce the following notations,
    \begin{equation*}
    \interleave x\interleave_l := \sqrt{\sum_{i,j\in I_l}|x_i-x_j|^2}, \quad \interleave v\interleave_l := \sqrt{\sum_{i,j\in I_l}|v_i-v_j|^2}.
  \end{equation*}
  According to \eqref{CS} and Section \ref{sec:2}, it is easy to get
  \begin{equation} \label{est-17}
    \interleave x\interleave_l \le 2 |I_l|^{\frac{1}{2}} \|x\|_2 \le 2 N^{\frac{1}{2}} \|x\|_2,\quad  \interleave v\interleave_l \le 2 |I_l|^{\frac{1}{2}} \|v\|_2 \le 2 N^{\frac{1}{2}} \|v\|_2,\quad  \left| \frac{d\interleave x\interleave_l}{dt} \right| \le \interleave v\interleave_l.
  \end{equation}
 Then we apply It\^o's formula to obtain that
  \begin{equation}\label{D32}
    \begin{split}
      d\interleave v\interleave_l^2 =&\ 2\sum_{i,j\in I_l} (v_i-v_j) \left( \frac{\lambda}{N}\sum_{k=1}^N \underbrace{\left(\psi(|x_i-x_k|)(v_k-v_i) - \psi(|x_j-x_k|)(v_k-v_j)\right)}_{=:\mathcal{A}_{ijk}} dt \right) \\
      &+2D\sum_{i,j\in I_l} (v_i-v_j)^2dW+ \sum_{i,j\in I_l} D^2 |v_i-v_j|^2 dt \\
      =&\ \frac{2\lambda}{N}\sum_{i,j,k\in I_l}(v_i-v_j) \mathcal{A}_{ijk}dt+  \frac{2\lambda}{N}\sum_{i,j\in I_l,k\notin I_l}(v_i-v_j)\mathcal{A}_{ijk}dt \\
      &+ 2D \interleave v\interleave_l^2 dW + D^2\interleave v\interleave_l^2 dt \\
      =:&\ Q_1dt + Q_2 dt + 2D \interleave v\interleave_l^2 dW + D^2\interleave v\interleave_l^2 dt.
    \end{split}
  \end{equation}
 For the term $Q_1$, as $i,j,k$ are all in the collision set $I_l$, we apply the antisymmetry property of the system \eqref{CS} and the fact that $\psi(\cdot)\ge\psi_*(\cdot)$ and the function $\psi_*(\cdot)$ is non-increasing, to obtain that on $\{|I_l| >1\}$,
  \begin{equation}\label{est-8}
    \begin{split}
      Q_1 &= -\frac{2\lambda|I_l|}{N} \sum_{i,j\in I_l} \psi(|x_i-x_j|) |v_i-v_j|^2 \\
      &\le -\frac{2\lambda}{N} \sum_{i,j\in I_l} \psi_*(|x_i-x_j|) |v_i-v_j|^2 \\
      &= -\frac{2\lambda}{N}\psi_*(\interleave x\interleave_l)\interleave v\interleave_l^2.
    \end{split}
  \end{equation}
  On the other hand, for every $r>0$, as $\psi$ is global Lipschitz on $[r,\infty)$ by the assumptions, we can define $L(r)$ to be the Lipschitz constant of $\psi$ over $[r,\infty)$, that is,
  \begin{equation}\label{lip-const}
    L(r):=\sup_{r\le r_1<r_2<\infty} \frac{|\psi(r_1)-\psi(r_2)|}{|r_1-r_2|}<\infty.
  \end{equation}
  Obviously, $L$ is non-increasing on $(0,\infty)$ and hence it can have singularity only possibly at $0$. Then by \eqref{est-14}, Cauchy--Schwarz inequality and the fact that $\psi(\cdot)\ge \psi_*:=\inf_{r>0}\psi(r)$, we have a.s. on $\{\tau^*\le T, |I_l| >1\}$ that
  \begin{equation}\label{est-9}
    \begin{split}
      Q_2 & =  \frac{2\lambda}{N} \sum_{i,j\in I_l,k\notin I_l} \left(\psi(|x_i-x_k|) - \psi(|x_j-x_k|)\right) (v_k-v_j)(v_i-v_j) \\
         &\quad -\frac{2\lambda}{N} \sum_{i,j\in I_l,k\notin I_l} \psi(|x_i-x_k|)|v_i-v_j|^2 \\
         & \le \frac{2\lambda}{N} L(\delta) \sum_{i,j\in I_l,k\notin I_l} |x_i- x_j| (v_k-v_j)(v_i-v_j) - \frac{2\lambda(N-|I_l|)\psi_*}{N} \interleave v\interleave_l^2 \\
         & \le \frac{4\lambda(N-|I_l|)}{N} L(\delta) \|v\|_2 \sum_{i,j\in I_l} |x_i- x_j||v_i-v_j| - 2\lambda\psi_* \interleave v\interleave_l^2 \\
         & \le 4\lambda L(\delta) \|v\|_2 \interleave x\interleave_l \interleave v\interleave_l - 2\lambda\psi_* \interleave v\interleave_l^2.
    \end{split}
  \end{equation}

  Next, 
  we define a Lyapunov functional as follows,
  \begin{equation*}
    \mathcal E_{\pm} = \interleave v\interleave_l \pm \frac{\lambda}{N} \Psi_*(\interleave x\interleave_l).
  \end{equation*}
  Note that a similar Lyapunov functional was introduced in \cite{HL09} for the deterministic C-S model. In order to derive the stochastic differential $d\interleave v\interleave_l$ from $d\interleave v\interleave_l^2$ in \eqref{D32}, we need to apply It\^o's formula to the function $f(r) = r^{1/2}$. But this function is not of class $C^2$ at the origin, while class $C^2$ is a requirement of It\^o's formula. To overcome, we claim here that a.s. on $\{\tau^*\le T, |I_l| >1\}$, it holds $\interleave v(t)\interleave_l \ne 0$ for all $t\in [0,\tau^*)$. Indeed, if this claim is false, then there exists a random time $\varsigma\in [0,\tau^*)$ such that $\P(\interleave v(\varsigma)\interleave_l = 0, \tau^*\le T, |I_l| >1 )>0$. On the event $\{\interleave v(\varsigma)\interleave_l=0, \tau^*\le T, |I_l| >1\}$, we apply the comparison principle to \eqref{D32} by using \eqref{est-8} and \eqref{est-9}, in a similar fashion as we did in Proposition 2.1, and it follows that a.s. $\interleave v(t)\interleave_l=0$ for all $t\in[\varsigma,\tau^*)$. This means, by the definition of $\interleave v\interleave_l$, that all particles in the set $I_l$ with $|I_l| >1$ will stay relatively still after the random time $\varsigma$ satisfying $\varsigma\le \tau^*\le T$, while they will collide at time $\tau^*$. This leads to a contradiction to the minimality of $\tau^*$ as the first collision time. The claim follows so that we can apply \eqref{D32} and It\^o's formula to obtain the following stochastic differential over $[0,\tau^*)$,
  \begin{equation}\label{est-10}
    \begin{split}
      d\mathcal E_{\pm} & = \frac{1}{2\interleave v\interleave_l} d\interleave v\interleave_l^2 - \frac{1}{8\interleave v\interleave_l^3} d[\interleave v\interleave_l^2,\interleave v\interleave_l^2] \pm \frac{\lambda}{N} d\Psi_*(\interleave x\interleave_l) \\
         & = \left[\frac{Q_1 + Q_2}{2\interleave v\interleave_l} \pm \frac{\lambda}{N} \psi_*(\interleave x\interleave_l) \frac{d\interleave x\interleave_l}{dt}\right] dt + D \interleave v\interleave_l dW.
    \end{split}
  \end{equation}
  We combine the fact that $\psi(\cdot)\ge \psi_*$ with \eqref{est-17}, \eqref{est-8} and \eqref{est-9} to have a.s. on $\{\tau^*\le T, |I_l| >1\}$ that
  \begin{equation*}
    \begin{split}
      \frac{Q_1 + Q_2}{2\interleave v\interleave_l} \pm \frac{\lambda}{N} \left[ \psi_*(\interleave x\interleave_l) - \psi_* \right] \frac{d\interleave x\interleave_l}{dt} &\le \frac{Q_1 + Q_2}{2\interleave v\interleave_l} + \frac{\lambda}{N} \left[ \psi_*(\interleave x\interleave_l) - \psi_* \right] \interleave v\interleave_l \\
      &\le 2\lambda L(\delta) \|v\|_2 \interleave x\interleave_l - \lambda\left( 1+\frac{1}{N} \right)\psi_* \interleave v\interleave_l,
    \end{split}
  \end{equation*}
  which implies
  \begin{equation}\label{D36}
    \begin{split}
      \frac{Q_1 + Q_2}{2\interleave v\interleave_l} \pm \frac{\lambda}{N} \psi_*(\interleave x\interleave_l) \frac{d\interleave x\interleave_l}{dt} &\le 2\lambda L(\delta) \|v\|_2 \interleave x\interleave_l + \lambda\left( 1+\frac{2}{N} \right)|\psi_*| \interleave v\interleave_l,
    \end{split}
  \end{equation}
  Now, we integrate \eqref{est-10} over $[0,t\wedge \sigma]$ for any stopping time $\sigma<\tau^*$ and apply \eqref{D36}, and obtain that a.s. on $\{\tau^*\le T, |I_l| >1\}$,
  \begin{equation}\label{est-15}
    \begin{split}
      &\interleave v(t\wedge \sigma)\interleave_l \pm \frac{\lambda}{N} \Psi_*(\interleave x(t\wedge \sigma)\interleave_l) \\
      \le&\ \interleave v(0)\interleave_l \pm \frac{\lambda}{N} \Psi_*(\interleave x(0)\interleave_l) + 2\lambda L(\delta) \int_0^{t\wedge \sigma} \|v(s)\|_2 \interleave x(s)\interleave_l ds \\
      &\  + \lambda\left( 1+\frac{2}{N} \right)|\psi_*| \int_0^{t\wedge \sigma} \interleave v(s)\interleave_l ds + \int_0^{t\wedge \sigma} D(s) \interleave v(s)\interleave_l dW(s).
    \end{split}
  \end{equation}
  According to Definition \ref{def-local}, there always exists a sequence of stopping times $\{\sigma_n\}_{n\ge1}$ such that $\sigma_n\uparrow\tau^*$ and $\sigma_n<\tau^*$ on $\{\tau^*\le T\}$. Moreover, \eqref{est-13} yields a.s. on $\{\tau^*\le T, |I_l| >1\}$ that for all $i,j\in I_l, i\ne j$, $\lim_{n\to\infty} |x_i(\sigma_n)-x_j(\sigma_n)| = 0$, and hence
  \begin{equation}\label{D38}
    \lim_{n\to\infty} \interleave x(\sigma_n)\interleave_l = 0.
  \end{equation}
  We then substitute $\sigma$ in \eqref{est-15} by $\sigma_n$ and remove the velocity term on the left hand side in \eqref{est-15} to obtain a.s. on the event $\{\tau^*\le T, |I_l| >1\}$ that
  \begin{equation}\label{D39}
    \begin{split}
      & \frac{\lambda}{N} \left|\Psi_*(\interleave x(\sigma_n)\interleave_l) \right| \\
      \le&\ \frac{\lambda}{N} \left| \Psi_*(\interleave x(0)\interleave_l) \right| + \interleave v(0)\interleave_l + 2\lambda L(\delta) \sup_{0\le t\le T} \int_0^{t\wedge \sigma_n} \|v(s)\|_2 \interleave x(s)\interleave_l ds \\
      &\ + \lambda\left( 1+\frac{2}{N} \right)|\psi_*| \sup_{0\le t\le T} \int_0^{t\wedge \sigma_n} \interleave v(s)\interleave_l ds \\
      &\ + \sup_{0\le t\le T} \left| \int_0^{t\wedge \sigma_n} D(s) \interleave v(s)\interleave_l dW(s) \right| \\
      =:&\ J_1+J_2+J^n_3+J^n_4+J^n_5.
    \end{split}
  \end{equation}

 \noindent $\bullet$ (Step 3). In this step, we will estimate $J_i$'s term by term as we did in the proof of Proposition \ref{P4.1}. Actually, we apply \eqref{D38} and \eqref{D39} as well as the assumption that $|\Psi_*(r)|\to\infty$ as $r\to0$, to have
  \begin{equation}\label{D40}
    \begin{split}
      &\ \P(\tau^*\le T, |I_l| >1 )\\
      =&\ \P\left( \tau^*\le T, |I_l| >1, \lim_{n\to\infty}\left|\Psi_*(\interleave x(\sigma_n)\interleave_l) \right| = \infty \right) \\
      \le&\ \P(J_1=\infty,|I_l| >1) + \P(J_2=\infty) + \P\left( \tau^*\le T, \liminf_{n\to\infty} J^n_3=\infty \right) \\
      &\ + \P\left( \tau^*\le T, \liminf_{n\to\infty} J^n_4=\infty \right) + \P\left(\tau^*\le T, \liminf_{n\to\infty} J^n_5=\infty\right).
    \end{split}
  \end{equation}
  For the term $J_1$, 
  we apply \eqref{est-17} to have
  \[ \min_{i\ne j}|x_i(0)-x_j(0)|\leq \interleave x(0)\interleave_l \le 2N^{\frac{1}{2}}\|x(0)\|_2.\]
  Observe that $\Psi_*$ only has singularity at 0 and maybe also at $\infty$ due to the assumptions. Then, with the same argument as \eqref{est-16}, we apply \eqref{est-17} and the collisionlessness of $\mu_0$ and the inner regularity of the distribution of $\|x(0)\|_2$ to obtain the following estimate,
  \begin{equation}\label{D41}
    \begin{split}
      \P(J_1=\infty,|I_l| >1) =&\ \P\left( \left| \Psi_*(\interleave x(0)\interleave_l) \right| = \infty, |I_l| >1 \right) \\
      \le&\ \P\left( \interleave x(0)\interleave_l = 0, |I_l| >1 \right) + \P\left( \interleave x(0)\interleave_l = \infty \right) \\
      \le&\ \P\left( \min_{i\ne j}|x_i(0)-x_j(0)|= 0 \right) + \P\left( \|x(0)\|_2 = \infty \right) \\
      =&\ 0.
    \end{split}
  \end{equation}
  For $J_2$, we use \eqref{est-17} and the inner regularity of the distribution of $\|v(0)\|_2$ to get
  \begin{equation}\label{D42}
    \P(J_2=\infty) \le \P\left( \|v(0)\|_2 = \infty \right) = 0.
  \end{equation}
  For $J_3^n$ on $\{\tau^*\le T\}$, we combine \eqref{est-17} and \eqref{est-2} to have
    \begin{align*}
      J_3^n & =2\lambda L(\delta) \sup_{0\le t\le T} \int_0^{t\wedge \sigma_n} \|v(s)\|_2 \interleave x(s)\interleave_l ds \\
      &\le 2\lambda L(\delta) \sup_{0\le t\le T} \int_0^{t\wedge \sigma_n} \|v(s)\|_2 \left( \interleave x(0)\interleave_l + \int_0^s \interleave v(r)\interleave_l dr \right) ds \\
      &\le 4N^{\frac{1}{2}}\lambda L(\delta) \sup_{0\le t\le T} \int_0^{t\wedge \sigma_n} \|v(s)\|_2 \left( \|x(0)\|_2 + \int_0^s \|v(r)\|_2 dr \right) ds \\
      &= 4N^{\frac{1}{2}}\lambda L(\delta) \sup_{0\le t\le T} \left[ \|x(0)\|_2 \int_0^{t\wedge \sigma_n} \|v(s)\|_2 ds + \frac{1}{2} \left( \int_0^{t\wedge \sigma_n} \|v(s)\|_2 ds \right)^2 \right] \\
      &\le 4N^{\frac{1}{2}}\lambda L(\delta) \left[ \|x(0)\|_2 \|v(0)\|_2 \int_0^T e^{\left(-\frac{1}{2}\int_0^s D^2(u) du + \int_0^s D(u)dW(u)\right)} ds \right.\\
      &\hspace{2.3cm}\left.+ \frac{1}{2} \|v(0)\|_2^2 \left(\int_0^T e^{\left(-\frac{1}{2}\int_0^s D^2(u) du + \int_0^s D(u)dW(u) \right)} ds\right)^2\right].
    \end{align*}
  Since $L$ only has possible singularity at $0$ as we have seen straight after \eqref{lip-const}, it follows from the fact that the random variable $\delta$ is positive on $\{\tau^*\le T\}$ that
  \begin{equation*}
    \P\left( \tau^*\le T, L(\delta) = \infty \right) \le \P\left(\tau^*\le T, \delta = 0 \right) = 0.
  \end{equation*}
  Then we apply a similar argument as \eqref{est-18} to obtain that
  \begin{equation}\label{D43}
    \begin{split}
      &\ \P\left(\tau^*\le T, \liminf_{n\to\infty} J_3^n=\infty\right) \\
      \le&\  \P\left( \tau^*\le T, L(\delta) = \infty \right) + \P\left( \|x(0)\|_2 = \infty \right) + \P\left( \|v(0)\|_2 = \infty \right) \\
      &\ +  \lim_{K\rightarrow+\infty}\P\left( \int_0^T \exp\left(-\frac{1}{2}\int_0^s D^2(u) du + \int_0^s D(u)dW(u) \right) ds \ge K \right)\\
      =&\ 0.
    \end{split}
  \end{equation}
  For the probability involving $J^n_4$, we derive in a similar fashion and use the fact that $|\psi_*|<\infty$,
  \begin{equation}\label{D2}
    \begin{split}
      &\ \P\left(\tau^*\le T, \liminf_{n\to\infty} J^n_4=\infty\right) \le \P\left( \tau^*\le T, \sup_{0\le t\le T} \int_0^{t\wedge \sigma_n} \|v(s)\|_2 ds = \infty\right) \\
      \le&\ \P\left( \int_0^T \exp\left(-\frac{1}{2}\int_0^s D^2(u) du + \int_0^s D(u)dW(u) \right) ds = \infty \right) = 0.
    \end{split}
  \end{equation}
  For the probability involving $J_5^n$, following the same way as \eqref{est-19}, we apply BDG inequality and \eqref{est-17} to derive that
  \begin{align*}
      &\ \P\left(\tau^*\le T, \liminf_{n\to\infty} J^n_5\ge K \right) \\
      \le&\ \left(\frac{D^*(T)}{K}\right)^{2p} \liminf_{n\to\infty} \E\left[ \left( \sup_{0\le t\le T}\left|\int_0^{t\wedge\sigma_n} \interleave v(s)\interleave_l dW(s)\right| \right)^{2p} \right] \\
      \le&\ \left(\frac{D^*(T)}{K}\right)^{2p} \liminf_{n\to\infty} \E\left[\left( \int_0^{T\wedge\sigma_n} \interleave v(s)\interleave_l^2 ds \right)^{p} \right] \\
      \le&\ \left(\frac{2ND^*(T)}{K}\right)^{2p} \E\left[\|v(0)\|_2^{2p}\left(\int_0^{T} \exp\left(-\int_0^s D^2(u) du + 2\int_0^s D(u)dW(u) \right) ds \right)^{p}\right].
  \end{align*}
 where $D^*(T)=\sup_{0\leq t\leq T}D(t)$ as in \eqref{est-19}. Hence, similar to \eqref{est-19}, we can find proper indexes $p,q$ and apply H\"older's inequality, Jensen's inequality and Fubini's theorem to yield
  \begin{equation}\label{D44}
    \lim_{K\to\infty} \P\left(\liminf_{n\to\infty} J^n_5\ge K \right) \le \lim_{K\to\infty} c(N,T,D^*(T),r,q) \left( \int_{\R^{2N}}\|v\|_2^\e\mu_0(dx,dv) \right)^{1/q} K^{-2p} = 0.
  \end{equation}
  Finally, we combine \eqref{D40}--\eqref{D44} to conclude that \eqref{D29} holds, i.e., $\P(\tau^*\le T, |I_l| >1 )=0$ for any $T>0$ and $l\in\{1,2,\cdots,N\}$, which finishes the proof.
\end{proof}

\section{Large time behavior}\label{sec:5}

We have defined conditional and unconditional flocking in mean in Definition \ref{def-cond-flocking} and \ref{def-flocking} respectively. In this section, we will investigate these large time behavior for the solution of the stochastic C-S model \eqref{CS}. In order to specify the assumptions for $\psi$, $D$ and the initial data to the large time behavior, we will assume throughout this section that the assumptions in Theorem \ref{local} and Theorem \ref{coll-avoid} are all in force, so that the system \eqref{CS} \emph{admits a global strong solution} $(x,v)$.

Recall that in Section \ref{sec:4}, we required $\psi$ to be sufficiently singular at the origin to avoid collisions. Now, for the large time behavior, the integrability of $\psi$ at infinity plays the role.
We have introduced two infimum notations in previous sections:
$$\psi_*:=\inf_{r\ge0}\psi(r), \quad \psi_*(r) := \inf_{0< s\le r}\psi(s).$$
Note that the second notation should not be confused with the first as it involves an independent variable.

Before the statements of the results, let us also recall some notions in probability theory. For a random variable $X$ and $1<q<\infty$, the $L^q$-norm of $X$ is defined by the quantity $\E(|X|^q)^{1/q}$, while the $L^\infty$-norm is defined by the essential supremum as follows,
\begin{equation*}
  \esssup X := \inf\{K\ge0: \P(|X|>K)=0\}.
\end{equation*}
The support of $X$ is defined to be the support of the distribution of it as a probability measure on $(\R,\B(\R))$. Obviously, if $X\in L^\infty(\Omega)$, or equivalently, $\esssup X<\infty$, then $X$ has compact support. It is well-known that if $X\in L^\infty(\Omega)$, then $\lim_{q\to\infty} \E(|X|^q)^{1/q} = \esssup X$. Therefore, without ambiguity, we will admit $q=\infty$ when we write $\E(|X|^q)^{1/q}$, and this quantity will indicate the $L^\infty$-norm of $X$.
Now, we first study the case that the $\psi$ has positive lower bound.
\begin{proposition}\label{flocking-1}
  Let $p\ge2$, and the assumptions in Theorem \ref{local} and Theorem \ref{coll-avoid} be satisfied. Assume that the first and second marginal distributions of the initial distribution $\mu_0$ in \eqref{initial-cond} both have finite moment, i.e., $\E ( \|x(0)\|_p) < \infty$ and $\E ( \|v(0)\|_p) < \infty$. Suppose that $\psi$ has positive lower bound, that is, $\psi_*>0$.
  Then the time-asymptotic flocking in mean occurs exponentially fast.
\end{proposition}

\begin{proof}
 According to the Theorem \ref{local} and Theorem \ref{coll-avoid}, the unique global strong solution is guaranteed. Now, to prove the velocity alignment, we use \eqref{apriori-2} and H\"older's inequality, as well as the independence of $v(0)$ and $W$ to derive
  \begin{equation}\label{est-3}
    \begin{split}
      \E \left( \|v(t)\|_p \right) &\le \E \left[ \|v(0)\|_p \exp\left(-\lambda\psi_* t -\frac{1}{2}\int_0^t D^2(s) ds + \int_0^t D(s)dW(s) \right) \right] \\
      &= \E \left( \|v(0)\|_p \right) \E \left[\exp\left(-\lambda\psi_* t -\frac{1}{2}\int_0^t D^2(s) ds + \int_0^t D(s)dW(s) \right) \right] \\
      &= \E \left( \|v(0)\|_p \right) e^{-\lambda\psi_* t} \\
      &\to 0,  \quad\text{as } t\to\infty.
    \end{split}
  \end{equation}
  For the group forming, we apply \eqref{apriori-1} and \eqref{est-3},
  \begin{equation*}
    \E \left( \|x(t)\|_p \right) \le \E \left( \|x(0)\|_p \right) + \int_0^t \E \left( \|v(s)\|_p \right) ds \le \E \left( \|x(0)\|_p \right) + \frac{1}{\lambda\psi_*} \E \left( \|v(0)\|_p \right) < \infty.
  \end{equation*}
  This ends the proof.
\end{proof}

The assumption $\psi_*>0$ is so strong that it requires an uniform strong effect even for far field. In general, this cannot be fulfilled, and a more applicable and natural setting is $\psi_*\ge0$. For instance, $\frac{1}{(1+r^2)^{\alpha/2}}$, $\frac{1}{r^\alpha}$ and $\frac{1}{\left(\log (1+r)\right)^\alpha}$ are all such communication weights. For this larger class of $\psi$ with $\psi_*\geq 0$, it is very difficult to control the far field, and thus we need more delicate estimates to gain some balance between the alignment force and the noise. In the following, we will assume the noise intensity to be constant or square integrable, and provide time-asymptotic analysis for these two cases respectively.


\subsection{Nonzero constant intensity $D$}

In this part, we will consider the simple case that $D(t)\equiv D$ is a nonzero constant. Actually, for more general case such that $D(t)$ is not square integrable, we can use the same arguments to obtain similar results. As discussed in the introduction, the noise with nonzero constant intensity will provide an uniform-in-time random effect to the velocity, which makes the alignment more difficult in expectation sense. However, the time integral of the exponential martingale will be finite a.s. in this case, and thus we start from the aggregation analysis and have the following conditional flocking results, where we recall that $\psi_*(r) := \inf_{0<s\le r}\psi(s)$.

\begin{theorem}[Conditional flocking for nonzero constant noise intensity]\label{cond-flock}
Let $p\ge2$, $D$ be a nonzero constant, and the assumptions in Theorem \ref{local} and Theorem \ref{coll-avoid} be fulfilled. 
Assume that $\E ( \|x(0)\|_p) < \infty$. Suppose $\psi_*(r) >0$ for all $r\ge0$, and there exists a constant $\beta>2$ such that $\psi$ has asymptotic structure at far field as below,
\begin{equation}\label{asmp-0}
    \lim_{r\to +\infty}\psi_*(r)r^\beta =\infty.
\end{equation}
Define the following event in $\F_0 \vee \F^W_\infty$,
\begin{equation}\label{asmp-2}
  A:=\left\{ \|x(0)\|_p \|v(0)\|_p \left[ \int_0^\infty \exp\left( -\frac{\beta D^2}{2(\beta-2)}s + \frac{\beta D}{\beta-2}W(s) \right) ds \right]^{1-\frac{2}{\beta}} < \frac{1}{4}\left(\frac{\beta\lambda}{2^{\beta+1}} \right)^{\frac{2}{\beta}} \right\}.
\end{equation}
Then the conditional aggregation in mean given $A$ emerges: for all $t\ge0$,
\begin{equation*}
  \E(\|x(t)\|_p | A) \le c(\psi,\E(\|x(0)\|_p)) < \infty.
\end{equation*}
Furthermore, we have the following two types of conditional flocking:
\begin{itemize}
  \item[(i).] If there exists $1<q<\infty$ such that $\E ( \|v(0)\|_p^q) < \infty$, then the conditional flocking in mean given $A$ (see Definition \ref{def-cond-flocking}) occurs algebraically fast: for all $0<\gamma<\frac{q-1}{\beta q}$ and all $t\geq 0$,
\begin{equation*}
  \E \left( \|v(t)\|_p | A\right) \le c\left( \lambda,\P(A),\E\left( \|x(0)\|_p \right),\E \left( \|v(0)\|_p^q \right)\right) (1+t)^{-\gamma}.
\end{equation*}
  \item[(ii).] If $\E ( \|v(0)\|_p) < \infty$ and $\esssup\,\|x(0)\|_p<\infty$, 
      then the conditional flocking in mean given $A$ occurs exponentially fast: for all $t\geq 0$,
\begin{equation*}
  \E \left( \|v(t)\|_p | A \right)\le c\left( \P(A),\esssup\,\|v(0)\|_p \right) e^{-\lambda c(\psi) t}.
\end{equation*}
\end{itemize}

\end{theorem}
\begin{remark}\label{remark-1}
  Before starting the proof, we make some remarks which help to understand the theorem.

  (i). It is easy to see that the infinite time integral $\int_0^\infty \exp(-\frac{\beta D^2}{2(\beta-2)}s + \frac{\beta D}{\beta-2}W(s)) ds$ appeared in \eqref{asmp-2} is almost surely finite, due to the growth of at most order $\sqrt{t\log\log t}$ of $|W(t)|$ as $t\to\infty$ (e.g., \cite[Theorem 2.9.23]{KS91}). Moreover, it was shown in \cite[Proposition 4.4.4]{Duf90} that this time integral obeys the inverse-gamma distribution $\mathrm{Inv}$-$\mathrm{Gamma}(\frac{\beta-2}{\beta}, -\frac{(\beta D)^2}{2(\beta-2)^2})$, which exactly indicates the assumptions $D\ne0$ and $\beta>2$ due to the requirements for the parameters of the inverse-gamma distribution. Hence the support of this integral as a random variable is the total interval $[0,\infty)$. It follows that the event $A$ defined in \eqref{asmp-2} does have positive probability, which is just the requirement in Definition \ref{R2.6}. Indeed, the inner regularity implies there exists $L>0$ such that $\P(\|x(0)\|_p \|v(0)\|_p<L)>0$, and then the independence of $(x(0),v(0))$ with $W$ yields
    \begin{equation*}
      \begin{split}
        \P(A) \ge&\ \P(\|x(0)\|_p \|v(0)\|_p<L) \\
        & \times\P\left( \int_0^\infty \exp\left( -\frac{\beta D^2}{2(\beta-2)}s + \frac{\beta D}{\beta-2}W(s) \right) ds < (4L)^{\frac{\beta}{2-\beta}} \left(\frac{2^{\beta+1}}{\beta\lambda} \right)^{\frac{2}{2-\beta}} \right) \\
        >&\ 0.
      \end{split}
    \end{equation*}
  Moreover, if $\P(x(0)=0\text{ or } v(0)= 0)<1$, then there exists $L>0$ such that $\P(\|x(0)\|_p \|v(0)\|_p \ge L)>0$. It follows that
    \begin{equation*}
      \begin{split}
        \P(A^c) \ge&\ \P(\|x(0)\|_p \|v(0)\|_p \ge L) \\
        & \times\P\left( \int_0^\infty \exp\left( -\frac{\beta D^2}{2(\beta-2)}s + \frac{\beta D}{\beta-2}W(s) \right) ds \ge (4L)^{\frac{\beta}{2-\beta}} \left(\frac{2^{\beta+1}}{\beta\lambda} \right)^{\frac{2}{2-\beta}} \right) \\
        >&\ 0,
      \end{split}
    \end{equation*}
  and hence $\P(A)<1$. That is, the flocking in mean in the above theorem is indeed conditional.

  (ii). The common cases $\psi=1$, $\psi(r) = r^{-\alpha}$ and $\psi(r) = \frac{1}{(1+r^2)^{\alpha/2}}$, with $\alpha>0$, evidently fulfill the assumption \eqref{asmp-0} for $\psi$ at far field. 
\end{remark}

\begin{proof}
We will first prove the aggregation, and then show the emergence of flocking. For notational simplicity, we denote
  \begin{equation}\label{E-4}
   X_t := \sup_{0\le s\le t} \|x(s)\|_p.
  \end{equation}

\noindent $\bullet$ (Conditional aggregation). We adapt the continuity approach in \cite{CS07A} to prove the emergence of aggregation. First, we use the comparison theorem for one-dimensional SDEs, as in the proof of Proposition \ref{regular} (in fact, here we need to adapt the comparison theorem in \cite[Theorem VI.1.1]{IW89} to the SDE with random drift, whereas this extension is easy to prove), and get as long as a global solution $(x,v)$ exists that, with probability one,
  \begin{equation}\label{E-5}
    \|v(t)\|_p \le \|v(0)\|_p \exp\left( -\lambda \int_0^t\psi(2\|x(s)\|_p) ds -\frac{1}{2} D^2t + DW(t) \right), \quad\forall t\ge0.
  \end{equation}
  Next, due to the assumption \eqref{asmp-0}, there exists $J>0$, such that $\psi(r)\ge\psi_*(r)\ge r^{-\beta}$ for all $r\in[J,\infty)$. Then, we fix a time $t\ge0$.
  On the event $\{X_t \ge J \}$, we simply follow the estimates \eqref{apriori-1} and \eqref{E-5} and apply H\"older's inequality to derive that,
    \begin{align}
      X_t &\leq \|x(0)\|_p+\int_0^t\|v(s)\|_pds\notag\\
      &\le \|x(0)\|_p + \|v(0)\|_p \int_0^t \exp\left(-\lambda \int_0^s\psi(2\|x(u)\|_p) du -\frac{D^2}{2}s + DW(s) \right) ds\notag \\
      &\le \|x(0)\|_p + \|v(0)\|_p \int_0^t \exp\left(-\lambda \psi_*(2X_t)s -\frac{D^2}{2}s + DW(s) \right) ds\notag \\
      &\le \|x(0)\|_p + \|v(0)\|_p \int_0^t \exp\left(-\lambda (2X_t)^{-\beta}s -\frac{D^2}{2}s + DW(s) \right) ds  \label{est-38}\\
      &\le \|x(0)\|_p + \|v(0)\|_p \left[ \int_0^\infty \exp\left( -\frac{\beta}{2}\lambda (2X_t)^{-\beta}s \right) ds \right]^{\frac{2}{\beta}} \notag\\
      &\hspace{3.5cm}\times\left[ \int_0^\infty \exp\left( -\frac{\beta D^2}{2(\beta-2)}s + \frac{\beta D}{\beta-2}W(s) \right) ds \right]^{\frac{\beta-2}{\beta}} \notag\\
      &= \|x(0)\|_p + \|v(0)\|_p \left[ \int_0^\infty \exp\left( -\frac{\beta D^2}{2(\beta-2)}s + \frac{\beta D}{\beta-2}W(s) \right) ds \right]^{\frac{\beta-2}{\beta}} \left(\frac{2^{\beta+1}}{\beta\lambda} \right)^{\frac{2}{\beta}} X_t^2 \notag\\
      &=: C + B X_t^2,\qquad \text{a.s. on } \{X_t\geq J\},\notag
    \end{align}
  where we used the fact that from the assumption that $\psi(r)\ge\psi_*(r)$ for all $r\ge0$ and $\psi_*(r)$ is non-increasing in $r$. As we have seen in Remark \ref{remark-1}.(i) that the infinite time integral $\int_0^\infty \exp(-\frac{\beta D^2}{2(\beta-2)}s + \frac{\beta D}{\beta-2}W(s)) ds$ is almost surely finite. 
  Thus, $C$ and $B$ are both well-defined random variables. If $B = \|v(0)\|_p = 0$, then $X_t \le C = \|x(0)\|_p$, that is,
  \begin{equation}\label{est-40}
    X_t \le C, \quad\text{a.s. on } \{X_t \ge J, \|v(0)\|_p = 0\}.
  \end{equation}
  Otherwise, if $B > 0$, then we define a quadratic function $G(z)$ as follows,
  \[G(z): = B z^2 - z + C.\]
  Then the graph of $G$ is a parabola opening upwards, with axis of symmetry locating at $z_* = \frac{1}{2B}>0$. Moreover, we simply have $A=\{1 - 4CB > 0\}$ where $A$ is the event defined in \eqref{asmp-2}. In other words, in the event $A$, the discriminant of $G$ is positive and hence, $G(z)$ has two nonnegative roots $z_{\pm} = \frac{1\pm\sqrt{1 - 4CB}}{2B}$.

  On the other hand, \eqref{est-38} can be rewritten as $G(X_t)\ge0$ for all $t\ge0$, and it is obvious that the map $t\to G(X_t)$ is continuous due to the continuity of the map $t\to X_t$. Now, as $X_0 = C \le \frac{1}{4B} < z_*=\frac{1}{2B}$, we obtain for each $t\ge0$ that,
  \begin{equation}\label{est-41}
    X_t\le z_- \quad\text{a.s. on } \{X_t \ge J, \|v(0)\|_p > 0\}\cap A.
  \end{equation}
 Then, for $z_-$, if we regard
  $$z_- = z_-(B) = \frac{1-\sqrt{1 - 4CB}}{2B}$$
  as a function of $B\in(0,\frac{1}{4C})$, it is easy to check that $z_-$ is increasing and hence $z_-< z_-(\frac{1}{4C}) = 2C$. Combining \eqref{est-40} and \eqref{est-41}, we conclude that for each $t\ge0$,
  $$X_t\le J \vee C \vee z_- \le J\vee 2C \le J+2C \quad\text{a.s. on } A.$$
  It follows from the continuity (or monotony) of $X_t$ that,
  \begin{equation}\label{est-39}
    X_t\le J+2C\quad \text{a.s. on } A, \quad\text{for all } t\ge0.
  \end{equation}
  Therefore, by the assumption $\E ( \|x(0)\|_p) < \infty$, we have
  \begin{equation*}
    \begin{split}
      \sup_{t\ge0} \E(\|x(t)\|_p | A) &\le \sup_{t\ge0} \E(X_t | A) \le J + 2\E(\|x(0)\|_p) < \infty.
    \end{split}
  \end{equation*}
  The aggregation or group forming follows.\newline

\noindent $\bullet$ (Conditional flocking algebraically fast). Now, according to \eqref{est-39}, we know the process $\{X_t\}$ is uniformly bounded by  a constant $J$ and the random variable $2\|x(0)\|_p$ on event $A$. Therefore, we apply \eqref{E-4}, \eqref{E-5} and \eqref{est-39}, and follow similar analysis as \eqref{est-38} to obtain for each $K>0$ that,
    \begin{align}
      \E \left( \|v(t)\|_p | A\right) \le&\ \E \left[ \|v(0)\|_p \exp\left(-\lambda \int_0^t\psi(2\|x(s)\|_p) ds -\frac{1}{2}D^2t+ DW(t) \right) \Bigg| A \right] \notag\\
      \le&\ \E \left[ \|v(0)\|_p \exp\left(-\lambda \psi_*(J+2\|x(0)\|_p) t -\frac{1}{2}D^2t+ DW(t) \right) \Bigg| A \right] \notag\\
      \le&\ \frac{1}{\P(A)} \E \left[ \|v(0)\|_p \exp\left(-\lambda \psi_*(J+2\|x(0)\|_p) t \right) \right] \E \left[ \exp\left( -\frac{1}{2}D^2t+ DW(t) \right) \right] \label{E-10}\\
      =&\ \frac{1}{\P(A)} \E \left[ \left( \ind_{\{ \|x(0)\|_p< K\}} + \ind_{\{ \|x(0)\|_p\ge K\}} \right) \|v(0)\|_p \exp\left(-\lambda \psi_*(J+2\|x(0)\|_p) t \right) \right] \notag\\
      \le&\ \frac{1}{\P(A)} \left[ \E \left( \|v(0)\|_p \right) \exp(-\lambda \psi_*(J+2K) t) + \E \left( \|v(0)\|_p ; \|x(0)\|_p\ge K \right) \right], \notag
    \end{align}
   where in the second and last inequalities we used the fact from the assumption that $\psi(r)\ge\psi_*(r)>0$ for all $r\ge0$ and $\psi_*(r)$ is non-increasing in $r$, and in the third inequality we used the independence of $v(0)$ and $W$. Then 
   we apply H\"older's inequality and Chebyshev's inequality to obtain that
   \begin{align}
     &\ \E \left( \|v(0)\|_p ; \|x(0)\|_p\ge K \right) \notag\\
     \leq&\ \left[ \E \left( \|v(0)\|_p^q \right) \right]^{1/q} \left[ \P\left( \|x(0)\|_p\ge K \right)\right]^{(q-1)/q} \label{E-11}\\
     \leq&\ \left[ \E \left( \|v(0)\|_p^q \right) \right]^{1/q} \left[ \frac{\E\left( \|x(0)\|_p \right)}{K} \right]^{(q-1)/q}. \notag
   \end{align}
   Now, we let the auxiliary parameter $K$ to be time-dependent as $K=K(t)=t^{\frac{\gamma q}{q-1}}$, and combine \eqref{E-10} and \eqref{E-11} together to obtain that
    \begin{align}
      \E \left( \|v(t)\|_p | A \right) \le &\ \frac{\left[ \E \left( \|v(0)\|_p^q \right) \right]^{1/q}}{\P(A)} \left[ \exp(-\lambda \psi_*(J+2K(t)) t) + \frac{1}{K(t)^{(q-1)/q}} \left[ \E\left( \|x(0)\|_p \right)\right]^{(q-1)/q} \right] \notag\\
     \le & \ \frac{\left[ \E \left( \|v(0)\|_p^q \right) \right]^{1/q}}{\P(A)} \left[ \exp\left(-\lambda \left(J+2t^{\frac{\gamma q}{q-1}} \right)^{-\beta} t\right) + \frac{1}{t^\gamma} \left[ \E\left( \|x(0)\|_p \right)\right]^{(q-1)/q} \right], \notag \\
     \le &\ c\left( \lambda, \P(A),\E\left( \|x(0)\|_p \right),\E \left( \|v(0)\|_p^q \right)\right) (1+t)^{-\gamma}, \notag
    \end{align}
where we used the assumption $\gamma<\frac{q-1}{\beta q}$ to observe that the exponential term at the RHS of the second inequality decays faster than $t^{-\gamma}$ as $t\to\infty$. This proves (i). \newline


\noindent $\bullet$ (Conditional flocking exponentially fast).  Finally, we consider the case $\esssup\,\|x(0)\|_p<\infty$. In this case, we can choose $K$ in \eqref{E-10} to be a sufficiently large constant so that
   \begin{equation*}
   \P \left( \|x(0)\|_p\ge K \right)=0.
   \end{equation*}
Then \eqref{E-10} reduce to,
\[\E \left( \|v(t)\|_p | A \right)\le \frac{\E \left( \|v(0)\|_p \right)}{\P(A)} \exp(-\lambda \psi_*(J+2K) t).\]
The exponentially fast emergence of the flocking follows.
\end{proof}

\begin{remark}
  (i). When $\esssup\,\|x(0)\|_p<\infty$, we can actually have the essentially uniform bound for $\|x(t)\|_p$ on $A$, due to \eqref{est-39}. 

  (ii). If $D\equiv 0$, one can easily see that the proof still works by simply letting $\beta=2$ and adopting the convention $\infty^0 = 1$ in \eqref{est-38}. Thus, the theorem also holds for $D\equiv 0$ and $\psi$ satisfying $\lim_{r\to +\infty}\psi_*(r)r^2 =\infty$ whereas the conditioning event is given by
  \begin{equation*}
    A := \left\{ \|x(0)\|_p \|v(0)\|_p < \frac{\lambda}{16} \right\}.
  \end{equation*}
  This is very similar to the conditional flocking result in \cite[Theorem 2.(iii)]{CS07A} for the deterministic C-S system, the difference on the assumption for $\psi$ and the conditioning event is due to the specialized estimates. In fact, the authors in \cite{CS07A} derived an algebraic equation of optional order from a similar argument with \eqref{est-38}, rather than a quadratic equation as we derived here.
\end{remark}

In Theorem \ref{cond-flock}, the conditioning event \eqref{asmp-2} is very important to generate the aggregation estimates even for long-range communication. Therefore, these estimates is far from optimal. Next, we will try to drop the condition \eqref{asmp-2} and have a better flocking estimates. As a compensation, we have to make $\psi$ decay slower at far field so that we can control the random effect from the noise. First of all, we see an unconditional alignment result.

\begin{lemma}[Unconditional alignment for nonzero constant noise intensity]\label{align-constant-D}
Let $p\ge2$, $D$ be a nonzero constant, and the assumptions in Theorem \ref{local} and Theorem \ref{coll-avoid} be fulfilled. Assume that $\|v(0)\|_p$ is uniformly integrable, that is,
\[\lim_{L\to\infty}\E( \|v(0)\|_p; \|v(0)\|_p\ge L ) = 0.\]
 Suppose $\psi_*(r) >0$ for all $r\ge0$, and there exists an $\e>\frac{1}{2}$ such that $\psi$ has slow decay structure at far field as below,
\begin{equation}\label{asspt-psi-1}
  \lim_{t\rightarrow+\infty}\psi_*(e^{\e D^2t })t =+\infty.
\end{equation}
Then the unconditional velocity alignment in mean occurs. 
\end{lemma}

\begin{remark}\label{R5.5}
%
  (i). A candidate for $\psi$ to fulfill the conditions above in Lemma \ref{align-constant-D} is $\psi(r) = |\log (1+r)|^{-\alpha}$ with $0<\alpha<1$, and the first condition is obviously weaker than the classical monotone decreasing assumption.

  (ii). Recall that if there exists $1<q\le\infty$ such that $[\E ( \|v(0)\|_p^q)]^{1/q} < \infty$ , then $\|v(0)\|_p$ is uniformly integrable \cite[Section 13.3]{Wil91}. In addition, the uniform integrability of $\|v(0)\|_p$ implies $\E ( \|v(0)\|_p) < \infty$. Thus, the assumption for $v(0)$ in the present lemma is weaker than that in Theorem \ref{cond-flock}.(i), but stronger than that in Proposition \ref{flocking-1} or Theorem \ref{cond-flock}.(ii).

  (iii). We make a final remark for the uniform integrability assumption of $\|v(0)\|_p$. This assumption implies that for every $\e>0$, there exists a $\delta>0$ such that, $\E(\|v(0)\|_p; A)<\e$ for every measurable $A$ satisfying $\P(A)<\delta$. Moreover, the converse is also true due to the inner regularity of the distribution of $\|v(0)\|_p$.
\end{remark}

\begin{proof}
  Firstly, we recall the notation in \eqref{E-4} that $X_t := \sup_{0\le s\le t} \|x(s)\|_p$, and the comparison theorem in \eqref{E-5} that, with probability one,
  \begin{equation}\label{E-14}
    \|v(t)\|_p \le \|v(0)\|_p \exp\left( -\lambda \int_0^t\psi(2\|x(s)\|_p) ds -\frac{1}{2} D^2 t + D W(t) \right), \,\forall t\ge0.
  \end{equation}
  Then, for each $K>0$, we have from \eqref{E-14} that
    \begin{align}
      \E \left( \|v(t)\|_p \right) \le&\ \E \bigg[ \left( \ind_{\{X_t< K\}} + \ind_{\{X_t\ge K\}} \right) \|v(0)\|_p \exp\left(-\lambda \int_0^t\psi(2\|x(s)\|_p) ds -\frac{D^2}{2}t+DW(t)\right) \bigg] \notag \\
      \le&\ \E\left( \|v(0)\|_p \right) \E\left[ \exp\left(-\lambda \psi_*(2K) t -\frac{D^2}{2}t+DW(t)\right) \right] \label{E-15} \\
      &\ + \E\left[ \|v(0)\|_p\exp\left(-\frac{D^2}{2}t +DW(t)\right); X_t\ge K \right] \notag\\
      =&\ \E\left( \|v(0)\|_p \right) e^{-\lambda \psi_*(2K)t} + \E\left[ \|v(0)\|_p\exp\left(-\frac{D^2}{2}t +DW(t)\right); X_t\ge K \right], \notag
  \end{align}
  where in the first equality we used the independence of $v(0)$ and $W$, and in the last inequality we used the fact from the assumption that $\psi(r)\ge\psi_*(r)>0$ for all $r\ge0$ and $\psi_*(r)$ is non-increasing in $r$. Now we define
  \begin{equation}\label{E-42}
    I(t) := \E\left[ \|v(0)\|_p\exp\left( -\frac{1}{2} D^2 t + D W(t)\right); X_t\ge K \right].
  \end{equation}
  The estimates \eqref{apriori-1} and \eqref{apriori-2} as well as $\psi>0$ yield that
  \begin{equation}\label{E-49}
    X_t \le \|x(0)\|_p + \|v(0)\|_p \int_0^t \exp\left( -\frac{1}{2} D^2 s + D W(s) \right) ds.
  \end{equation}
  Hence, we have for all $\theta\in(0,1)$ that,
  \begin{equation*}
    \begin{split}
      \left\{ X_t\ge K \right\} \subset&\ \left\{ \|x(0)\|_p \ge \frac{K}{2} \right\} \cup \left\{ \|v(0)\|_p \ge \left(\frac{K}{2}\right)^{1-\theta} \right\} \\
      &\ \cup\left\{ \int_0^t \exp\left( -\frac{1}{2} D^2 s + D W(s) \right) ds \ge \left(\frac{K}{2}\right)^\theta \right\}.
    \end{split}
  \end{equation*}
  According to this, we have
    \begin{align}
      I(t) \le&\ \E\left[ \|v(0)\|_p\exp\left( -\frac{1}{2} D^2 t + D W(t) \right); \|x(0)\|_p \ge \frac{K}{2} \right] \notag\\
      &\ + \E\left[ \|v(0)\|_p\exp\left( -\frac{1}{2} D^2 t + D W(t)\right); \|v(0)\|_p \ge \left(\frac{K}{2}\right)^{1-\theta} \right] \label{E-43}\\
      &\ + \E\left[ \|v(0)\|_p\exp\left( -\frac{1}{2} D^2 t + D W(t)\right); \int_0^t \exp\left( -\frac{1}{2} D^2 s + D W(s) \right) ds \ge \left(\frac{K}{2}\right)^\theta \right] \notag\\
      =:&\ I_1(t) + I_2(t) + I_3(t). \notag
    \end{align}
  We first estimate $I_3$. Using H\"older's inequality and Chebyshev's inequality, as well as the independence of $v(0)$ with $W$, we obtain that for every $a>1$,
  \begin{equation*}
    \begin{split}
      I_3(t) &\le \E\left( \|v(0)\|_p \right) \left( \E\left[ \exp\left( -\frac{a}{2}D^2t + aDW(t) \right)\right] \right)^{1/a} \\
      &\quad\times\left[ \P \left( \int_0^t \exp\left( -\frac{1}{2} D^2 s + D W(s) \right) ds \ge \left(\frac{K}{2}\right)^\theta \right) \right]^{1/a'} \\
      &\le \E\left( \|v(0)\|_p \right) \exp\left( \frac{a-1}{2}D^2 t \right) \left[ \left(\frac{K}{2}\right)^{-\theta} \int_0^t \E\left( \exp\left( -\frac{1}{2} D^2 s + D W(s) \right) \right) ds \right]^{1/a'} \\
      &= \E\left( \|v(0)\|_p \right) \left(\frac{2^\theta t e^{\frac{a}{2}D^2t}}{K^\theta}\right)^{1/a'},
    \end{split}
  \end{equation*}
  where we denote as before $a'$ the conjugate of $a$.
%
 Now, we take $K(t) = \frac{1}{2}e^{\e D^2 t}$ where $\e>\frac{1}{2}$ is the constant mentioned in the assumption for $\psi$, and choose the positive constant $\theta$ and $a$ to satisfy the following conditions,
 \[\theta\in(0,1),\quad a>1,\quad \e\theta-\frac{a}{2}>0.\]
As $\e>\frac{1}{2}$, we can always choose $\theta$ and $a$ to be sufficiently closed to $1$ so that above inequalities hold. Then using the observation in Remark \ref{R5.5}.(ii) that $\E\left( \|v(0)\|_p \right) <\infty$, we have the following estimate for the term $I_3(t)$,
  \begin{equation}\label{E-16}
   \lim_{t\rightarrow+\infty} I_3(t) \le \E\left( \|v(0)\|_p \right) \lim_{t\rightarrow+\infty} \left(\frac{4^\theta t}{e^{(\e\theta-\frac{a}{2})D^2 t}}\right)^{1/a'} = 0.
  \end{equation}
%
%
 Next, for the term $I_1(t)$, we use the independence of $(x(0),v(0))$ with $W$ to obtain that,
  \begin{equation}\label{E-44}
    \begin{split}
      I_1(t) &= \E\left[ \exp\left( -\frac{1}{2}D^2 t + D W(t) \right) \right] \E\left( \|v(0)\|_p; \|x(0)\|_p \ge \frac{K(t)}{2} \right) \\
      &= \E\left( \|v(0)\|_p; \|x(0)\|_p \ge \frac{K(t)}{2} \right).
    \end{split}
  \end{equation}
  Hence, the uniform integrability assumption for $\|v(0)\|_p$ and Remark \ref{R5.5}.(iii) yields the vanishing of $I_1(t)$ as $t$ goes to infinity. The term $I_2$ can be treated in the same way, and thus we have
   \begin{equation}\label{E-17}
   \lim_{t\rightarrow+\infty}I_1(t)=\lim_{t\rightarrow+\infty}I_2(t)=0.
   \end{equation}
 Finally, according to the assumption \eqref{asspt-psi-1} for $\psi$, we know that
 \begin{equation}\label{E-18}
  \lim_{t\to\infty} \E\left( \|v(0)\|_p \right) e^{-\lambda \psi_*(2K(t))t} = 0.
 \end{equation}
  Now combining \eqref{E-15}, \eqref{E-16}, \eqref{E-17} and \eqref{E-18},
  we conclude by letting $t\to\infty$ in both sides of \eqref{E-15} that,
  $$\lim_{t\to\infty} \E(\|v(t)\|_p) = 0.$$
  This proves the velocity alignment.
\end{proof}

\begin{remark}
  One can easily modify the proof of this lemma to adapt to the situation $D\equiv 0$. Indeed, as \eqref{E-49} now reduces to $X_t \le \|x(0)\|_p + t\|v(0)\|_p$, the case becomes much easier. We only need to let $K=K(t) = t^{1+\e}$ for some $\e>0$ and replace the paragraph from \eqref{E-43} to \eqref{E-17} by
  \begin{equation*}
    \begin{split}
      I(t) &\le \E\left( \|v(0)\|_p; \|x(0)\|_p + t\|v(0)\|_p \ge K \right) \\
      &\le \E\left( \|v(0)\|_p; \|x(0)\|_p \ge t^{1+\e}/2 \right) + \E\left( \|v(0)\|_p; \|v(0)\|_p \ge t^\e/2 \right) \\
      &\to 0,
    \end{split}
  \end{equation*}
  as $t\to+\infty$, where the last convergence is due to the uniform integrability of $\|v(0)\|_p$ as well as Remark \ref{R5.5}.(iii). Consequently, this lemma is still valid for the case $D\equiv 0$ if we replace the assumption \eqref{asspt-psi-1} by
  \begin{equation*}
    \lim_{t\rightarrow +\infty}\psi_*(t^{1+\e})t =+\infty, \quad \text{for some } \e>0.
  \end{equation*}
  As we will see soon, this is exactly the assumption for $\psi$ in Lemma \ref{align-integrable} in the next subsection, where the unconditional alignment for square integrable intensity is considered. For this reason, we do not attempt to cover the case of constant zeros intensity in the statement of the present lemma, and we just leave this case to Lemma \ref{align-integrable}. The same situation also applies for the uncondition flocking results in Theorem \ref{uncond-flock-constant-D} and \ref{flocking-integrable}.
\end{remark}

Now, we are ready to introduce the unconditional flocking results for the constant intensity. In order to yield the flocking estimates, we need the integrability of the velocity expectation, and thus a stronger communication weight is required. Moreover, we will see from the proof that the emergence of the flocking is actually at least algebraically fast.

\begin{theorem}[Unconditional flocking for nonzero constant noise intensity]\label{uncond-flock-constant-D}
Let $p\ge2$, $D$ be a nonzero constant, and the assumptions in Theorem \ref{local} and Theorem \ref{coll-avoid} be fulfilled. 
Assume that $\E ( \|x(0)\|_p) < \infty$ and $[\E\left( \|v(0)\|_p^q \right)]^{1/q}<\infty$ for some $1<q\le\infty$. Suppose $\psi_*(r) >0$ for all $r\ge0$, and there exist $\e>\frac{1}{2}$ and $\eta>\frac{1}{\lambda}$ such that $\psi$ has slow decay structure at far field as below,
\begin{equation}\label{E-19}
   \lim_{t\rightarrow +\infty}\left(\psi_*(e^{\e D^2 t})t - \eta \log t\right)=+\infty.
\end{equation}
Then the unconditional flocking in mean (defined in Definition \ref{def-flocking}) emerges algebraically fast.
\end{theorem}

\begin{remark}
It is easy to check the function $\psi(r) = |\log (1+r)|^{-\alpha}$ with $0<\alpha<1$ mentioned  in Remark \ref{R5.5} also satisfies the assumption \eqref{E-19} above.
\end{remark}

\begin{proof}
  It follows from the discussion in Remark \ref{R5.5}.(ii) that all the requirements in Lemma \ref{align-constant-D} are satisfied under the assumptions in the present theorem. This directly yields the emergence of alignment. Then, we only need to prove the aggregation, i.e.,
  \begin{equation}\label{E-20}
  \sup_{t\ge0} \E(\|x(t)\|_p) <\infty.
  \end{equation}

  For this purpose, we first observe from the assumption \eqref{E-19} that there exists constants $C, T>0$, such that,
  \begin{equation}\label{E-41}
    \psi_*(e^{\e D^2 t})t \ge \eta \log t + C, \quad  \text{for all } t\ge T.
  \end{equation}
  Next, we apply \eqref{E-14} and the independence between the initial data and the Brownian motion $W$, we find that
  \[\E(\|v(t)\|_p) \le \E( \|v(0)\|_p), \quad  \text{for all } t\ge0.\]
  Then, combining the above estimate with \eqref{apriori-1} and \eqref{E-15}, we obtain,
  \begin{equation}\label{E-21}
  \begin{aligned}
    &\ \sup_{t\ge0} \E(\|x(t)\|_p) \\
    \le&\ \E \left( \|x(0)\|_p \right) + \int_0^T \E \left( \|v(s)\|_p \right) ds + \int_T^\infty \E \left( \|v(s)\|_p \right) ds\\
    \le&\ \E \left( \|x(0)\|_p \right) + T\E \left( \|v(0)\|_p \right)  + \int_T^\infty \E \left( \|v(s)\|_p \right) ds\\
    \le&\ \E \left( \|x(0)\|_p \right) + T\E \left( \|v(0)\|_p \right) + \left( \E \left( \|v(0)\|_p \right) \int_T^\infty e^{-\lambda \psi_*(2K(s))s} ds +  \int_T^\infty I(s) ds\right),
    \end{aligned}
  \end{equation}
 where $T$ is the constant in \eqref{E-41}, the function $I$ in the last inequality is defined in \eqref{E-42}. As $\E \left( \|x(0)\|_p \right)$ and $ \E \left( \|v(0)\|_p \right) $ are both finite by the assumption, we can verify \eqref{E-20}  by showing the following finiteness due to \eqref{E-21},
   \begin{equation*}
    \int_T^\infty e^{-\lambda \psi_*(2K(s))s} ds < \infty, \quad \int_T^\infty I(s) ds < \infty
  \end{equation*}
  for some function $K=K(t)$. Now, we take $K(t) = \frac{1}{2}e^{\e D^2t}$, and apply \eqref{E-41} to have the following estimate,
 \begin{align}
   \int_T^\infty e^{-\lambda \psi_*(2K(s))s} ds\leq   \int_T^\infty e^{-\lambda C-\lambda\eta \log s} ds = e^{-\lambda C} \int_T^\infty \frac{1}{s^{\lambda \eta}}ds< \infty\label{E-23},
  \end{align}
  where the last inequality is due to the assumption that $\eta>\frac{1}{\lambda}$. To estimate the time integral of $I$, we recall the estimates in Lemma \ref{align-constant-D} and split $I$ into three parts as in \eqref{E-43},
  \begin{equation}\label{E-24}
  I(t)\le I_1(t)+I_2(t)+I_3(t).
  \end{equation}
  Now, as $\e>\frac{1}{2}$, we can follow the proof of Lemma \ref{align-constant-D} to choose $\theta\in(0,1)$ and $a>1$ so that $\e\theta-\frac{a}{2}>0$. Then, it follows from \eqref{E-43} and \eqref{E-16} by applying H\"older's inequality and Chebyshev's inequality, as well as the independence of $v(0)$ with $W$ that
  \begin{align}
  \int_T^\infty I_1(s) ds &= \int_T^\infty \E\left( \|v(0)\|_p; \|x(0)\|_p \ge \frac{e^{\e D^2t}}{4} \right) ds \notag \\
  &\le \int_T^\infty \left[\E\left( \|v(0)\|_p^q \right)\right]^{\frac{1}{q}} \left[\frac{4\E\left( \|x(0)\|_p \right)}{e^{\e D^2t}}\right]^{\frac{1}{q'}} ds< \infty, \notag \\
  \int_T^\infty I_2(s) ds &= \int_T^\infty \E\left( \|v(0)\|_p; \|v(0)\|_p \ge \left( \frac{e^{\e D^2t}}{4} \right)^{1-\theta} \right) ds \notag \\
  &\le \int_T^\infty \left[\E\left( \|v(0)\|_p^q \right)\right]^{\frac{1}{q}} \left[\frac{4^{1-\theta}\E\left( \|v(0)\|_p \right)}{e^{\e(1-\theta) D^2t}}\right]^{\frac{1}{q'}} ds < \infty, \notag\\
  \int_T^\infty I_3(s) ds &\le \E\left( \|v(0)\|_p \right) \int_T^\infty \left(\frac{4^\theta t}{e^{(\e\theta-\frac{a}{2})D^2 t}}\right)^{1/a'} ds < \infty, \notag
  \end{align}
  where $q'$ is the conjugate of $q$. Combining these three estimates with \eqref{E-21}, \eqref{E-23} and \eqref{E-24}, the emergence of aggregation \eqref{E-20} follows. Moreover, due to the last three inequalities and the estimate \eqref{E-23}, it is clear that the flocking occurs at least algebraic fast. This finishes the proof.
\end{proof}

\subsection{Square integrable intensity $D(t)$}\label{subsec:5-2}
In this part, we will discuss the case when the noise intensity $D(t)$ is varying and square integrable over $t\in[0,\infty)$. In this case, the time integral of the exponential martingale is only well defined in finite time, and thus the arguments in Theorem \ref{cond-flock} to obtain the conditional flocking no longer applies. In the following, we will show the unconditional flocking in mean (defined in Definition \ref{def-flocking}) for a class of strong long-range communication weights.

\begin{lemma}[Unconditional alignment for square integrable noise intensity]\label{align-integrable}
Let $p\ge2$, $\int_0^\infty D^2(s) ds <\infty$, and the assumptions in Theorem \ref{local} and Theorem \ref{coll-avoid} be fulfilled. 
Assume that $\|v(0)\|_p$ is uniformly integrable, that is, $\lim_{L\to\infty}\E( \|v(0)\|_p; \|v(0)\|_p\ge L ) = 0$. Suppose $\psi_*(r) >0$ for all $r\ge0$, and there exists $\e>0$ such that the communication $\psi$ has asymptotic structure at far field as below,
    \begin{equation}\label{E-26}
      \lim_{t\rightarrow +\infty}\psi_*(t^{1+\e})t =+\infty.
    \end{equation}
Then the unconditional velocity alignment in mean occurs. 
\end{lemma}

\begin{remark}\label{alignment-special}
  (i). Two candidates for $\psi$ to fulfill the assumptions in Lemma \ref{align-integrable} are the two common cases $\psi(r) = r^{-\alpha}$ and $\psi(r) = \frac{1}{(1+r^2)^{\alpha/2}}$ with $0<\alpha<1$. So in the case $D\equiv0$ which is clearly square integrable, if we take the deterministic initial data as in Remark \ref{R2.6}, then this lemma merely recovers the unconditional alignment results in \cite[Proposition 4.1, Proposition 4.3]{HL09}, where the requirement for the two common cases is instead $0\le\alpha\le1$. The gap is of course due to the way we handle the random initial data (which is distributed not necessarily of compact support) and the noise in a general setting.

  (ii). One can exactly follow the lines of Lemma \ref{align-constant-D} to proof this lemma. But we will not do so, because we need a more appropriate estimate than \eqref{E-15} when proving the algebraically fast flocking in mean in the next theorem.
\end{remark}
%

\begin{proof}
Similar as in \eqref{E-5} or \eqref{E-14}, we know that as long as a global solution $(x,v)$ exists that, with probability one, for all $t\ge0$,
  \begin{equation}\label{E-27}
    \|v(t)\|_p \le \|v(0)\|_p \exp\left( -\lambda \int_0^t\psi(2\|x(s)\|_p) ds -\frac{1}{2}\int_0^t D^2(s) ds + \int_0^t D(s)dW(s) \right).
  \end{equation}
  For notational simplicity, we denote as before that $X_t := \sup_{0\le s\le t} \|x(s)\|_p$. We also denote
  \begin{equation*}
    M_t := \int_0^t D(s)dW(s).
  \end{equation*}
  Then $M=\{M_t\}_{t\ge0}$ is a martingale due to the local boundedness of $D$, and its quadratic variation is $[M]_t = \int_0^t D^2(s) ds$. As we said in the previous remark, we will not use the counterpart of the estimate \eqref{E-15}, even though it works. Instead, we apply \eqref{E-27} to obtain that for each $K>0$ and $L>0$,
  \begin{align}
      \E \left( \|v(t)\|_p \right) \le &\ \E \bigg[ \left( \ind_{\|v(0)\|_p< L\}} + \ind_{\|v(0)\|_p\ge L\}} \right) \notag\\
      & \quad\ \times \|v(0)\|_p \exp\left(-\lambda \int_0^t\psi(2\|x(s)\|_p) ds -\frac{D^2}{2}t+DW(t)\right) \bigg] \notag\\
      \le&\ L \E \left[ \exp\left(-\lambda \int_0^t\psi(2\|x(s)\|_p) ds -\frac{D^2}{2}t+DW(t)\right) \right] \notag\\
      &\ + \E\left[ \|v(0)\|_p\exp\left(-\frac{D^2}{2}t +DW(t)\right); \|v(0)\|_p\ge L \right] \notag\\
      =&\ L \E \bigg[ \left( \ind_{\{X_t< K\}} + \ind_{\{X_t\ge K\}} \right) \notag\\
      & \qquad\ \times \exp\left(-\lambda \int_0^t\psi(2\|x(s)\|_p) ds -\frac{D^2}{2}t+DW(t)\right) \bigg] \label{E-29}\\
      &\ + \E\left( \|v(0)\|_p; \|v(0)\|_p\ge L \right) \E\left[ \exp\left(-\frac{D^2}{2}t +DW(t)\right) \right] \notag\\
      \le&\ L \E\left[ \exp\left(-\lambda \psi_*(2K) t -\frac{D^2}{2}t+DW(t)\right) \right] \notag\\
      &\ + L\E\left[ \exp\left(-\frac{D^2}{2}t +DW(t)\right); X_t\ge K \right] \notag\\
      &\ + \E\left( \|v(0)\|_p; \|v(0)\|_p\ge L \right) \notag\\
      =&\ L e^{-\lambda \psi_*(2K)t} + L\E\left[ \exp\left(-\frac{D^2}{2}t +DW(t)\right); X_t\ge K \right] \notag\\
      &\ + \E\left( \|v(0)\|_p; \|v(0)\|_p\ge L \right),\notag
  \end{align}
  Now we let $K$ to be time-dependent as $K(t) = \frac{1}{2}t^{1+\e}$ with $\e>0$ the constant in the assumptions, and obtain from \eqref{E-26} that,
  \begin{equation}\label{E-30}
  \lim_{t\to\infty}e^{-\lambda \psi_*(2K(t))t} =  \lim_{t\to\infty}e^{-\lambda \psi_*\left(t^{1+\e}\right)t}=0.
  \end{equation}
 Next, we need to estimate the second term in \eqref{E-29}. In a similar fashion as \eqref{E-42} and \eqref{E-43}, if we denote
  \begin{equation}\label{E-45}
    \hat I(t) := \E\left[ \exp\left( -\frac{1}{2}[M]_t + M_t \right); X_t\ge K \right],
  \end{equation}
  then we have
  \begin{equation}\label{E-46}
    \begin{split}
      \hat I(t) \le&\ \E\left[ \exp\left( -\frac{1}{2}[M]_t + M_t \right); \|x(0)\|_p \ge \frac{K}{2} \right] \\
      &\ + \E\left[ \exp\left( -\frac{1}{2}[M]_t + M_t \right); \|v(0)\|_p \ge \left(\frac{K}{2}\right)^{1-\theta} \right] \\
      &\ + \E\left[ \exp\left( -\frac{1}{2}[M]_t + M_t \right); \int_0^t \exp\left( -\frac{1}{2}[M]_s + M_s \right) ds \ge \left(\frac{K}{2}\right)^\theta \right] \\
      =:&\ \hat I_1(t) + \hat I_2(t) + \hat I_3(t).
    \end{split}
  \end{equation}
  We first estimate $\hat I_3$. Actually, by using H\"older's inequality and Chebyshev's inequality, we obtain that for every $a>1$,
  \begin{equation*}
    \begin{split}
      \hat I_3(t) &\le \left( \E\left[ \exp\left( -\frac{a}{2}[M]_t + aM_t \right)\right] \right)^{1/a} \left[ \P \left( \int_0^t \exp\left( -\frac{1}{2}[M]_s + M_s \right) ds \ge \left(\frac{K}{2}\right)^\theta \right) \right]^{1/a'} \\
      &\le \exp\left( \frac{a-1}{2}[M]_t \right) \left[ \left(\frac{K}{2}\right)^{-\theta} \int_0^t \E\left( \exp\left( -\frac{1}{2}[M]_s + M_s \right) \right) ds \right]^{1/a'} \\
      &= \left(\frac{2^\theta t e^{\frac{a}{2}[M]_t}}{K^\theta}\right)^{1/a'},
    \end{split}
  \end{equation*}
  where as before $a'$ denotes the conjugate of $a$. As $[M]_\infty=\int_0^\infty D^2(s) ds <\infty$ and we already let $K(t) = \frac{1}{2}t^{1+\e}$, we may choose $\theta\in(0,1)$ such that $\theta(1+\e)-1>0$ and have the following limit,
  \begin{equation}\label{E-31}
   \lim_{t\rightarrow+\infty} \hat I_3(t) \le \lim_{t\rightarrow+\infty} \left(\frac{4^\theta e^{\frac{a}{2}[M]_\infty}}{t^{\theta(1+\e)-1}}\right)^{1/a'}= 0.
  \end{equation}
  On the other hand, in a similar way as \eqref{E-44} and \eqref{E-17}, using the independence of $(x(0),v(0))$ with $W$ as well as the inner regularity of the distribution of $\|x(0)\|_p$ and $\|v(0)\|_p$, we have
  \begin{equation}\label{E-32}
  \lim_{t\rightarrow+\infty}\hat I_1(t)=  \lim_{t\rightarrow+\infty}\hat I_2(t)= 0.
  \end{equation}
  Now combining \eqref{E-30}, \eqref{E-31} and \eqref{E-32}, we let $t\to\infty$ and $L\to\infty$ successively in both sides of \eqref{E-29}, and use the uniform integrability assumption for $\|v(0)\|_p$ to conclude that
  $$\lim_{t\to\infty} \E(\|v(t)\|_p) = 0.$$
  This shows the velocity alignment.
\end{proof}
Similar as the previous subsection, we need relative faster decay of the velocity expectation to yield the aggregation estimates. Therefore, we need a stronger communication $\psi$ and have the following flocking theorem.

\begin{theorem}[Unconditional flocking for square integrable noise intensity]\label{flocking-integrable}
Let $p\ge2$, $\int_0^\infty D^2(s) ds <\infty$, and the assumptions in Theorem \ref{local} and Theorem \ref{coll-avoid} be fulfilled. 
Assume that $\E ( \|x(0)\|_p) < \infty$ and $[\E\left( \|v(0)\|_p^q \right)]^{1/q}<\infty$ for some $1<q\le\infty$. Suppose $\psi_*(r) >0$ for all $r\ge0$, and there exist $\e>2$ and $\eta> \frac{q}{\lambda(q-1)}$ such that the communication $\psi$ has asymptotic structure at far field as below,
\begin{equation*}
  \lim_{t\rightarrow +\infty}\left(\psi_*(t^{1+\e})t - \eta \log t\right) =+\infty.
\end{equation*}
Then the unconditional flocking in mean (defined in Definition \ref{def-flocking}) emerges algebraically fast.
\end{theorem}

\begin{remark}
  (i). Two candidates for $\psi$ to fulfill the assumptions in Theorem \ref{flocking-integrable} are the two common cases $\psi(r) = r^{-\alpha}$ and $\psi(r) = \frac{1}{(1+r^2)^{\alpha/2}}$ with $0<\alpha<\frac{1}{3}$. Clearly, the requirements on $\psi$ becomes stronger than Lemma \ref{align-integrable}, so that when restricted to deterministic case, the present theorem has more gap to the unconditional flocking results in \cite[Proposition 4.1, Proposition 4.3]{HL09}, where they only requires $0\le\alpha\le1$ for the two common cases of $\psi$. But as we have said in Lemma \ref{alignment-special}.(i), this is expectable due to our dedicated techniques.

  (ii). If $\psi$ is of the typical form $\psi(r) = r^{-\alpha}$ or $\psi(r) = \frac{1}{(1+r^2)^{\alpha/2}}$, we can get a more sophisticated result. In fact, we will obtain the unconditional flocking in mean provided $0<\alpha<\frac{q-1}{2q-1}$, by utilizing the upper concave envelope. See \ref{App-B}.
\end{remark}

\begin{proof}
  Due to Lemma \ref{align-integrable}, we only need to prove the group forming, i.e., $\sup_{t\ge0} \E(\|x(t)\|_p) <\infty$.
To this end, we first observe from the assumption \eqref{E-19} that there exists constants $C, T>0$, such that,
  \begin{equation}\label{E-35}
    \psi_*(t^{1+\e})t \ge \eta \log t + C, \quad  \text{for all } t\ge T.
  \end{equation}
  We use a same argument as in Theorem \ref{uncond-flock-constant-D}. Then by virtue of the estimate \eqref{E-29}, to complete the proof it suffices to find two time-dependent functions $K$ and $L$ such that the following hold
  \begin{equation}\label{E-34}
    \int_T^\infty L(s)e^{-\lambda \psi_*(2K(s))s} ds < \infty, \ \int_T^\infty L(s)\hat I(s) ds < \infty, \ \int_T^\infty \E\left( \|v(0)\|_p; \|v(0)\|_p\ge L(s) \right)ds <\infty,
  \end{equation}
  where $\hat I$ is the function defined in \eqref{E-45}. Here in the current situation, we take
  $$K(t) = \frac{1}{2}t^{1+\e} \text{ with } \e>2, \quad\text{and } L(t) = t^\delta \text{ with }\frac{1}{q-1}<\delta<\lambda \eta-1.$$
  Note that the constant $\delta$ does exist since $\lambda \eta > \frac{q}{q-1}$ due to the assumption.
  Then \eqref{E-35} yields
  \begin{equation}\label{E-36}
   \int_T^\infty L(s)e^{-\lambda \psi_*(2K(s))s} ds\leq   \int_T^\infty s^\delta e^{-\lambda C-\lambda\eta \log s} ds = e^{-\lambda C} \int_T^\infty \frac{1}{s^{\lambda \eta-\delta}}ds<+\infty.
  \end{equation}
  For the estimate of the time integral of $\hat I$, we recall the estimate \eqref{E-46} and split $\hat I$ into three parts, namely,
  \begin{equation}\label{E-37}
    \hat I(t) \le \hat I_1(t)+\hat I_2(t)+\hat I_3(t).
  \end{equation}
  Now, as $\e>2$, it is possible for us to find the positive constants $\theta$ and $a$ satisfying the following properties,
 \[\theta\in(0,1),\quad a'>1,\quad \theta(1+\e)-1>a', \quad (1-\theta)(1+\e)>1,\]
  where $a'$ is the conjugate of $a$. Then, due to above choice of constants and \eqref{E-31}, we have the following estimate of the time integral of $\hat I_3$,
  \begin{equation}\label{E-38}
 \int_T^\infty \hat I_3(s)ds \le \int_T^\infty \left(\frac{4^\theta e^{\frac{a}{2}[M]_\infty}}{s^{\theta(1+\e)-1}}\right)^{1/a'}ds<+\infty.
  \end{equation}
 Next, for $\hat I_1$, we simply use Chebyshev's inequality to get
  \begin{equation}\label{E-39}
    \int_T^\infty \hat I_1(s) ds \le \int_T^\infty \P\left( \|x(0)\|_p \ge \frac{s^{1+\e}}{4} \right) ds \le \int_T^\infty \frac{4\E\left( \|x(0)\|_p \right)}{s^{1+\e}} ds < \infty.
  \end{equation}
  In a similar fashion, since $(1-\theta)(1+\e)>1$, we have the following for $\hat I_2$,
  \begin{equation}\label{E-40}
    \int_T^\infty \hat I_2(s) ds \le \int_T^\infty \P\left( \|v(0)\|_p \ge \left( \frac{s^{1+\e}}{4} \right)^{1-\theta} \right) ds \le \int_T^\infty \frac{4^{1-\theta}\E\left( \|v(0)\|_p \right)}{s^{(1+\e)(1-\theta)}} ds < \infty.
  \end{equation}
  For the last integral in \eqref{E-34}, we use Chebyshev's inequality and H\"older's inequality to have
  \begin{equation}\label{E-47}
    \begin{split}
      \int_T^\infty \E\left( \|v(0)\|_p; \|v(0)\|_p\ge L(s) \right)ds &\le \int_T^\infty \left[\E\left( \|v(0)\|_p^q \right)\right]^{1/q} \left[\P\left( \|v(0)\|_p\ge L(s) \right)\right]^{1/q'} ds \\
      &\le \int_T^\infty \left[\E\left( \|v(0)\|_p^q \right)\right]^{1/q} \left[\frac{\E\left( \|v(0)\|_p^q \right)}{L(s)^q}\right]^{1/q'} ds \\
      &= \E\left( \|v(0)\|_p^q \right) \int_T^\infty s^{-\delta(q-1)} ds \\
      &< \infty,
    \end{split}
  \end{equation}
  where in the last inequality we used that $\delta(q-1)>1$ due to the selection of $\delta$. Finally, we combine \eqref{E-36}, \eqref{E-37}, \eqref{E-38}, \eqref{E-39}, \eqref{E-40} and \eqref{E-47} to achieve \eqref{E-34}. The emergence of flocking in mean follows. Moreover, it is easy to see from those estimates that the flocking occurs algebraically fast.
\end{proof}

\section{Conclusions and discussion}\label{sec:6}
\setcounter{equation}{0}
In this paper, we study the singular C-S model with a multiplicative noise, and provide the well posedness, collision-avoidance and large time behavior results. Actually, we prove the existence of a unique solution before the first collision for the system \eqref{CS}. Then we consider higher order singular case and prove the collision-avoidance and thus obtain the global existence of the solution. For the large time behavior, the emergence of flocking can be obtained similarly as the regular case if the communication function has a positive lower bound. This kind of communications acts like a uniform damping so that the velocity will tends to the mean asymptotically. While for the communication function with zero lower bound, the original Lyapunov functional approach fails since the particle positions lose the compactness in general. Therefore, we instead apply the estimates of the exponential martingale to capture the decay of the velocity. Based on different structures of the communications, we obtain conditional flocking and unconditional flocking respectively. Then, there are two natural and challenging problems left: the first is how to define the solution and describe the dynamics after the collisions happen if they are inevasible; the second is how to describe the asymptotic cluster formation when $\psi$ decay fast at far field. These two problems are fundamental and will appear in many other systems with singular potential or interactions, and we will continue to study them in our future works.

\appendix

\section{Non-flocking in two particle system}\label{App-A}
\setcounter{equation}{0}
In this section, we will provide a simple example such that flocking will not emerge in the sense of Definition \ref{def-cond-flocking} on particular events, which is mentioned in Remark \ref{R2.6}. We assume that the C-S system has only two particles moving on the real line $\R$. Then the system \eqref{CS} reads
\begin{equation}\label{CS-2}\left\{
  \begin{aligned}
    dx_1 &= v_1 dt, \quad dv_1 = \frac{\lambda}{2} \psi(|x_1-x_2|)(v_2-v_1)dt + D v_1 dW, \\
    dx_2 &= v_2 dt, \quad dv_2 = \frac{\lambda}{2} \psi(|x_1-x_2|)(v_1-v_2)dt + D v_2 dW. 
  \end{aligned} \right.
\end{equation}
If we set
\begin{equation*}
  x:= x_1 - x_2, \quad v:= v_1 - v_2,
\end{equation*}
then the pair $(x,v)$ satisfies the following system
\begin{equation}\label{est-42}
  dx = v dt, \quad dv = -\lambda \psi(|x|)vdt + D v dW,
\end{equation}

\begin{proposition}
  Let $D$ be a constant. Assume that $\psi>0$. The system \eqref{CS-2} does not flock conditioning on the following event
  \begin{equation}\label{event-1}
    \left\{ v(0) \ge \lambda \int_{x(0)}^\infty \psi(|x|) dx \right\},
  \end{equation}
  if the event has nonzero probability.
\end{proposition}

\begin{proof}
  According to \eqref{est-42} we have
  \begin{equation*}
    dv = -\lambda \psi(|x|)dx + D v dW.
  \end{equation*}
  It follows by integrating over time interval $[0,t]$ that
  \begin{equation*}
    v(t) - v(0) = -\lambda \int_{x(0)}^{x(t)} \psi(|y|)dy + D \int_0^t v(s) dW(s).
  \end{equation*}
  Denote the event \eqref{event-1} by $A$. Then on this event, we have
  \begin{equation}\label{est-43}
    v(t) \ge \lambda \int_{x(t)}^\infty \psi(|y|)dy + D \int_0^t v(s) dW(s).
  \end{equation}
  Observe that $\int_0^\cdot v(s) dW(s)$ is a martingale due to the estimate \eqref{apriori-2}. Taking expectation conditioning on $A$ on both sides of \eqref{est-43} and noting the fact that $A$ is independent of $W$, we approach
  \begin{equation}\label{est-44}
    \E(v(t) | A) \ge \lambda \E\left( \int_{x(t)}^\infty \psi(|y|)dy \Bigg| A \right) + D \E\left( \int_0^t v(s) dW(s) \right) = \lambda \E\left( \int_{x(t)}^\infty \psi(|y|)dy \Bigg| A \right).
  \end{equation}
  Now, assume in contrast that the conditional flocking given $A$ occurs, that is,
  \begin{equation*}
    \lim_{t\to\infty} \E(v(t)| A) \to 0 \quad\text{and}\quad \sup_{t\ge0} \E(|x(t)| | A) = M <\infty.
  \end{equation*}
  Then we choose $K>M$ and use Chebyshev's inequality for conditional expectation to proceed \eqref{est-44},
  \begin{equation*}
    \begin{split}
      \E(v(t) | A) &\ge \lambda \E\left( \int_{x(t)}^\infty \psi(|y|)dy; |x(t)| \le K \Bigg| A \right) \\
      &\ge \lambda \int_K^\infty \psi(|y|)dy \P\left( |x(t)| \le K | A \right) \\
      &\ge \lambda \int_K^\infty \psi(|y|)dy \left( 1-\frac{\E\left( |x(t)| | A \right)}{K} \right) \\
      &\ge \lambda \int_K^\infty \psi(|y|)dy \left( 1-\frac{M}{K} \right),
    \end{split}
  \end{equation*}
  where the RHS of the last inequality is a positive constant independent of $t$. This leads to a contradiction to $\lim_{t\to\infty} \E(v(t)| A) \to 0$.
\end{proof}

\section{Flocking for two typical communications}\label{App-B}

When the communication weight is of the form $\psi(r) = r^{-\alpha}$ at far field, we can obtain better estimates as follows.
\begin{proposition}
  Let $p\ge2$, $\int_0^\infty D^2(s) ds <\infty$, and the system \eqref{CS} admit a global strong solution $(x,v)$. Assume that $\E ( \|x(0)\|_p) < \infty$ and $[\E\left( \|v(0)\|_p^q \right)]^{1/q}<\infty$ for some $1<q\le\infty$. Suppose $\psi$ is of the form $\psi(r) = r^{-\alpha}$ or $\psi(r) = \frac{1}{(1+r^2)^{\alpha/2}}$, $r>0$, with
  \begin{equation}\label{asspt-alpha}
    0<\alpha< \frac{q-1}{2q-1}.
  \end{equation}
  Then the unconditional flocking in mean emerges algebraically fast.
\end{proposition}
\begin{proof}
  We will only prove the case $\psi(r) = r^{-\alpha}$ as the other is trivially similar. Obviously, the assumption \eqref{asspt-alpha} for $\alpha$ yields $0<\alpha<\frac{1}{2}$. Hence, we have already shown the velocity alignment in Lemma \ref{align-integrable} and Remark \ref{alignment-special}.(i) for this case. It is only left to prove the group forming.

  To this end, recall in \eqref{E-4} that we denote $X_t := \sup_{0\le s\le t} \|x(s)\|_p$. Then in a similar fashion as \eqref{E-29}, we use \eqref{E-27} and H\"older's inequality to have, for each $a>1$,
  \begin{equation}\label{est-33}
    \begin{split}
      \E \left( \|v(t)\|_p \right) &\le \E \left[ \left( \ind_{\|v(0)\|_p< L\}} + \ind_{\|v(0)\|_p\ge L\}} \right) \|v(0)\|_p \exp\left(-\lambda \psi(2X_t) t -\frac{1}{2}[M]_t + M_t \right) \right] \\
      &\le L\E \left[ \exp\left(-\lambda \psi(2X_t) t -\frac{1}{2}[M]_t + M_t \right) \right] \\
      & \quad+ \E \left[ \|v(0)\|_p \exp\left( -\frac{1}{2}[M]_t + M_t \right); \|v(0)\|_p\ge L \right] \\
      &\le L \left( \E\left[ \exp\left( -\frac{a}{2}[M]_t + aM_t \right)\right] \right)^{1/a} \left( \E e^{-a'\lambda \psi(2X_t) t} \right)^{1/a'} \\
      & \quad+ \E \left( \|v(0)\|_p ; \|v(0)\|_p\ge L \right) \\
      &\le L \exp\left( \frac{a-1}{2}[M]_t \right) \left[ \E (F_t(X_t)) \right]^{1/a'} + \E \left( \|v(0)\|_p ; \|v(0)\|_p\ge L \right).
    \end{split}
  \end{equation}
  where $M_t := \int_0^t D(s)dW(s)$, $a'$ denotes the conjugate of $a$, and for each $t\ge0$ and $a>1$ we denote
  $$F_t^{(a)}(r):=e^{-a'\lambda \psi(2 r) t} = e^{-a'\lambda (2 r)^{-\alpha} t}, r\ge0.$$
  We define an auxiliary function by
  \begin{equation*}
    \hat F_t^{(a)}(r) :=
    \begin{cases}
      2r(e\alpha a'\lambda t)^{-1/\alpha}, & 0\le r \le \frac{1}{2} (\alpha a'\lambda t)^{1/\alpha}, \\
      e^{-a'\lambda (2 r)^{-\alpha} t}, & r > \frac{1}{2} (\alpha a'\lambda t)^{1/\alpha}.
    \end{cases}
  \end{equation*}
  Then it is easy to see that $\hat F_t^{(a)}(r)$ is increasing and concave in $r$ and $F_t^{(a)}\le \hat F_t^{(a)}$ for each $t\ge0$. Indeed, $\hat F_t^{(a)}$ is the smallest concave function that lies upon $F_t^{(a)}$, which is the so called \emph{upper concave envelope}. Applying Jensen's inequality we have
  \begin{equation}\label{est-34}
    \E \left(F_t^{(a)}(X_t)\right) \le \E \left( \hat F_t^{(a)}(X_t) \right) \le \hat F_t^{(a)}(\E (X_t)).
  \end{equation}
  On the other hand, \eqref{E-49} implies
  \begin{equation}\label{est-35}
    \E(X_t) \le \E(\|x(0)\|_p) + t\E(\|v(0)\|_p).
  \end{equation}
  Combining \eqref{est-33}, \eqref{est-34}, \eqref{est-35} as well as the assumption $[M]_\infty = \int_0^\infty D^2(s) ds <\infty$, we get
  \begin{equation}\label{E-48}
    \begin{split}
      \E \left( \|v(t)\|_p \right) &\le L \exp\left( \frac{a-1}{2}[M]_\infty \right) \left[ \hat F_t^{(a)}(\E(\|x(0)\|_p) + t\E(\|v(0)\|_p)) \right]^{1/a'} \\
      & \quad + \E \left( \|v(0)\|_p ; \|v(0)\|_p\ge L \right).
    \end{split}
  \end{equation}
  Since $\alpha<\frac{1}{2}$, there exists $T>0$ such that whenever $t\in[T,\infty)$, $\E(\|x(0)\|_p) + t\E(\|v(0)\|_p) \le \frac{1}{2} (\alpha a'\lambda t)^{1/\alpha}$ and hence
  \begin{equation*}
    \hat F_t^{(a)}(\E(\|x(0)\|_p) + t\E(\|v(0)\|_p)) = \frac{2(\E(\|x(0)\|_p) + t\E(\|v(0)\|_p))}{(e\alpha a'\lambda t)^{1/\alpha}} \to 0
  \end{equation*}
  as $t\to\infty$. This recover the velocity alignment $\lim_{t\to\infty} \E(\|v(t)\|_p) = 0$ by letting $t\to\infty$ and $L\to\infty$ successively in both sides of \eqref{E-48}. Finally, we take $L$ to be time-dependent as $L(t)=t^\delta$, and choose $a>1$ and $\delta>\frac{1}{q-1}$ such that
  $$\frac{1}{a'}\left(\frac{1}{\alpha}-1\right) -\delta >1.$$
  This is achievable since $\frac{1}{\alpha}-1 >\frac{1}{q-1}+1$ by the assumption.
  Then it follows from \eqref{E-48} and \eqref{E-47} that
  \begin{equation*}
    \begin{split}
      \int_T^\infty \E \left( \|v(t)\|_p \right) dt &\le \exp\left( \frac{a-1}{2}[M]_\infty \right) \int_T^\infty L(t) \left[ \frac{2(\E(\|x(0)\|_p) + t\E(\|v(0)\|_p))}{(e\alpha a'\lambda t)^{1/\alpha}} \right]^{1/a'} dt \\
      & \quad + \int_T^\infty \E \left( \|v(0)\|_p ; \|v(0)\|_p\ge L(t) \right) dt \\
      &\le c\left(\alpha,\lambda,q,a,[M]_\infty,\E\left(\|x(0)\|_p\right), [\E\left( \|v(0)\|_p^q \right)]^{1/q} \right) \int_T^\infty \left( t^{\delta-\frac{1}{a'}(\frac{1}{\alpha}-1)} + t^{-\delta(q-1)} \right) dt \\
      &<\infty.
    \end{split}
  \end{equation*}
  The emergence of algebraically fast flocking in mean follows in a similar fashion as \eqref{E-21}.
\end{proof}

\paragraph{Acknowledgements.}
The work of X. Zhang is supported by the National Natural Science Foundation of China (Grant No. 11801194), Hubei Key Laboratory of Engineering Modeling and Scientific Computing, and National Natural Science Foundation of China (Grant No. 11971188). The work of Q. Huang is supported by FCT, Portugal, project PTDC/MAT-STA/28812/2017. We would like to thank Dr.~Shaung Chen, Dr.~Zhao Xu and Ms.~Zibo Wang for helpful discussions, and also thank the reviewers for their thoughtful comments and efforts towards improving our manuscript.


\end{document}